\documentclass[a4paper,oneside,11pt]{article}
\usepackage{amsmath}
\usepackage{amssymb}
\usepackage{amsthm}
\usepackage{MnSymbol}
\usepackage{enumerate}
\usepackage{enumitem}
\usepackage{bm}
\usepackage[margin = 3cm]{geometry}
\usepackage{color}
\usepackage{hyperref,url}

\newcommand{\N}{\mathbb{N}}
\newcommand{\R}{\mathbb{R}}

\newcommand{\dHaus}{\, \mathrm{d} \mathcal{H}^{d-1} \,}
\newcommand{\dd}{\, \mathrm{d} \,}
\newcommand{\dx}{\, \mathrm{d}x\,}
\newcommand{\dt}{\, \mathrm{d}t\,}
\newcommand{\dr}{\, \mathrm{d}r\,}

\newcommand{\pd}{\partial}
\newcommand{\eps}{\varepsilon}

\newcommand{\abs}[1]{\left| #1 \right|}
\newcommand{\norm}[1]{\| #1 \|}
\newcommand{\bignorm}[1]{\left\| #1 \right\|}
\newcommand{\inner}[2]{\langle #1 , #2 \rangle}

\newcommand{\Laplace}{\Delta}

\newcommand{\mean}[1]{\overline{#1}}
\renewcommand{\div}{\, \mathrm{div}\,}

\newtheorem{thm}{Theorem}[section]
\newtheorem{lemma}[thm]{Lemma}

\newtheorem{remark}{Remark}[section]

\newtheorem{defn}{Definition}[section]
\newtheorem{assump}{Assumption}[section]
\numberwithin{equation}{section}

\begin{document}

\title{Well-posedness of a Cahn--Hilliard system modelling tumour growth with chemotaxis and active transport}

\author{Harald Garcke \footnotemark[1] \and Kei Fong Lam \footnotemark[1]}

\date{\today}

\maketitle

\renewcommand{\thefootnote}{\fnsymbol{footnote}}
\footnotetext[1]{Fakult\"at f\"ur Mathematik, Universit\"at Regensburg, 93040 Regensburg, Germany
({\tt \{Harald.Garcke, Kei-Fong.Lam\}@mathematik.uni-regensburg.de}).}

\begin{abstract}
We consider a diffuse interface model for tumour growth consisting of a Cahn--Hilliard equation with source terms coupled to a reaction-diffusion equation.  The coupled system of partial differential equations models a tumour growing in the presence of a nutrient species and surrounded by healthy tissue.  The model also takes into account transport mechanisms such as chemotaxis and active transport.  We establish well-posedness results for the tumour model and a variant with a quasi-static nutrient.  It will turn out that the presence of the source terms in the Cahn--Hilliard equation leads to new difficulties when one aims to derive a priori estimates.  However, we are able to prove continuous dependence on initial and boundary data for the chemical potential and for the order parameter in strong norms.
\end{abstract}



\section{Introduction}\label{sec:Intro}
Several new diffuse interface models for tumour growth have been introduced recently in \cite{article:GarckeLamSitkaStyles}.  Amongst them is a Cahn--Hilliard equation coupled with a reaction-diffusion equation for a nutrient species.  The model equations are given as
\begin{subequations}\label{Intro:CHNutrient}
\begin{alignat}{2}
\pd_{t}\varphi  &= \div (m(\varphi) \nabla \mu) + (\lambda_{p} \sigma - \lambda_{a}) h(\varphi) && \text{ in } \Omega \times (0,T), \label{Intro:varphi} \\
\mu & = A\Psi'(\varphi) - B \Laplace \varphi - \chi_{\varphi} \sigma && \text{ in } \Omega \times (0,T), \label{Intro:mu} \\
\pd_{t}\sigma  & = \div (n(\varphi) (\chi_{\sigma} \nabla \sigma - \chi_{\varphi} \nabla \varphi)) -\lambda_{c} \sigma h(\varphi) && \text{ in } \Omega \times (0,T),   \label{Intro:sigma} \\
0 & = \nabla \varphi \cdot \nu = \nabla \mu \cdot \nu && \text{ on } \Gamma \times (0,T), \\
n(\varphi) \chi_{\sigma} \nabla \sigma \cdot \nu & = K (\sigma_{\infty} - \sigma) && \text{ on } \Gamma \times (0,T).
\end{alignat}
\end{subequations}
Here, $\Omega \subset \R^{d}$ is a bounded domain with boundary $\Gamma := \pd \Omega$, $\sigma$ denotes the concentration of an unspecified chemical species that serves as a nutrient for the tumour, $\varphi \in [-1,1]$ denotes the difference in volume fractions, with $\{\varphi = 1\}$ representing unmixed tumour tissue, and $\{\varphi = -1\}$ representing the surrounding healthy tissue, and $\mu$ denotes the chemical potential for $\varphi$.

In the system \eqref{Intro:CHNutrient}, $A$, $B$, and $K$ denote positive constants, $m(\varphi)$ and $n(\varphi)$ are positive mobilities for $\varphi$ and $\sigma$, respectively, $\Psi(\cdot)$ is a potential with two equal minima at $\pm 1$, $\sigma_{\infty}$ denotes a nutrient supply on the boundary $\Gamma$, and $h(\varphi)$ is an interpolation function with $h(-1) = 0$ and $h(1) = 1$.  The simplest example is $h(\varphi) = \frac{1}{2}(1+\varphi)$.

The non-negative constants $\lambda_{p}$, $\lambda_{a}$ represent the proliferation rate and the apoptosis rate of the tumour cells, respectively, and the non-negative constant $\lambda_{c}$ represents the consumption rate of the nutrient.  Here we note that these are only active in the tumour regions, and the healthy tissue does not proliferate, or consume nutrient or undergo apoptosis.

We denote $\chi_{\sigma} > 0$ as the diffusivity of the nutrient, and $\chi_{\varphi} \geq 0$ can be seen as a parameter for transport mechanisms such as chemotaxis and active uptake.  To see this, we note that in \eqref{Intro:varphi} and \eqref{Intro:sigma}, the fluxes for $\varphi$ and $\sigma$ are given by
\begin{align*}
\bm{q}_{\varphi} & := -m(\varphi) \nabla \mu = -m(\varphi) \nabla (A \Psi'(\varphi) - B \Laplace \varphi - \chi_{\varphi} \sigma), \\
\bm{q}_{\sigma} & := -n(\varphi)\nabla (\chi_{\sigma} \sigma - \chi_{\varphi}  \varphi),
\end{align*}
respectively.  The term $m(\varphi) \nabla (\chi_{\varphi} \sigma)$ in $\bm{q}_{\varphi}$ drives the cells in the direction of increasing $\sigma$, i.e., towards regions of high nutrient, and thus it models the chemotactic response towards the nutrient.  Meanwhile, the term $n(\varphi) \nabla (\chi_{\varphi} \varphi)$ in $\bm{q}_{\sigma}$ drives the nutrient to regions of high $\varphi$, i.e., to the tumour cells, which indicates that the nutrient is moving towards the tumour cells.  Note that $\nabla \varphi$ is non-zero only in the vicinity of the interface between the tumour cells and the healthy tissue, and thus this term only contributes significantly near the tumour interface.  In \cite{article:GarckeLamSitkaStyles}, the authors interpreted this term as the mechanisms that actively transport nutrient into the tumour colony, and establish a persistent nutrient concentration difference between the different cell compartments even against the nutrient concentration gradient.  The term ``active transport'' is used in the biological sense that some kind of mechanism is required to maintain the transport, which is in contrast to passive transport processes such as diffusion driven only by the concentration gradient.

We note that in \eqref{Intro:CHNutrient}, the mechanism of chemotaxis and active transport are connected via the parameter $\chi_{\varphi}$.  To ``decouple'' the two mechanisms, we introduce the following choice for the mobility $n(\varphi)$ and diffusion coefficient $\chi_{\sigma}$.  For a positive constant $\eta > 0$ and a positive mobility $D(\varphi)$, consider
\begin{align}\label{activetransportchoice}
n(\varphi) = \eta D(\varphi) \chi_{\varphi}^{-1}, \quad \chi_{\sigma} = \eta^{-1} \chi_{\varphi}.
\end{align}
Then, the corresponding fluxes for $\varphi$ and $\sigma$ are now given as
\begin{equation}\label{activetransportflux}
\begin{aligned}
\bm{q}_{\varphi} & := -m(\varphi) \nabla (A \Psi'(\varphi) - B \Laplace \varphi - \chi_{\varphi} \sigma), \\
\bm{q}_{\sigma} & := -D(\varphi)\nabla (\sigma - \eta \varphi),
\end{aligned}
\end{equation}
where the parameter $\chi_{\varphi}$ controls the effects of chemotaxis, and the parameter $\eta$ controls the effects of active transport.

We introduce the free energy $N$ for the nutrient as
\begin{align}\label{choice:Nvarphisigma}
N(\varphi, \sigma) = \frac{\chi_{\sigma}}{2} \abs{\sigma}^{2} + \chi_{\varphi} \sigma(1-\varphi),
\end{align}
and its partial derivatives with respect to $\sigma$ and $\varphi$ are given as
\begin{align}\label{choice:Npartialderi}
N_{,\sigma} = \chi_{\sigma} \sigma + \chi_{\varphi}(1-\varphi), \quad N_{,\varphi} = - \chi_{\varphi} \sigma.
\end{align}
Note that, by the boundary condition $\nabla \varphi \cdot \nu = 0$ on $\Gamma$, and the definition of $N_{,\sigma}$ \eqref{choice:Npartialderi}, we have
\begin{align*}
\nabla N_{,\sigma} \cdot \nu = \chi_{\sigma} \nabla \sigma \cdot \nu - \chi_{\varphi} \nabla \varphi \cdot \nu = \chi_{\sigma} \nabla \sigma \cdot \nu \text{ on } \Gamma.
\end{align*}
Thus, by testing \eqref{Intro:sigma} with $N_{,\sigma}$, \eqref{Intro:mu} with $\pd_{t}\varphi$, \eqref{Intro:varphi} with $\mu$, and summing the resulting equations, one can show the following formal energy identity is satisfied,
\begin{equation}\label{Intro:energyid}
\begin{aligned}
& \frac{\dd}{\dt} \int_{\Omega}  \left [ A \Psi(\varphi) + \frac{B}{2} \abs{\nabla \varphi}^{2} + \frac{\chi_{\sigma}}{2} \abs{\sigma}^{2} + \chi_{\varphi} \sigma(1-\varphi) \right ] \dx \\
& \quad + \int_{\Omega} m(\varphi) \abs{\nabla \mu}^{2} + n(\varphi) \abs{\nabla N_{,\sigma}}^{2} \dx + \int_{\Gamma} K N_{,\sigma}(\sigma - \sigma_{\infty}) \dHaus \\
& \quad + \int_{\Omega} - \mu (\lambda_{p} \sigma - \lambda_{a}) h(\varphi) + \lambda_{c} \sigma h(\varphi) N_{,\sigma} \dx = 0,
\end{aligned}
\end{equation}
where $\mathcal{H}^{d-1}$ is the $(d-1)$-dimensional Hausdorff measure.  To derive useful a priori estimates from \eqref{Intro:energyid} we face a number of obstacles:
\begin{enumerate}
\item the presence of source terms $\mu h(\varphi)(\lambda_{a} - \lambda_{p} \sigma) + N_{,\sigma} \lambda_{c} \sigma h(\varphi)$ deprives \eqref{Intro:energyid} of a Lyapunov structure, i.e., an inequality of the form $\frac{\dd}{\dt} V \leq \alpha V$, for $\alpha \geq 0$ and a suitable function $V$;
\item the term $\sigma(1-\varphi)$ in the nutrient free energy $N(\varphi, \sigma)$ can have a negative sign;
\item the presence of triple products $\mu \sigma h(\varphi)$ and $\sigma h(\varphi) N_{,\sigma}$.
\end{enumerate}
One way to control the triple products with the usual $H^{1}$-regularity expected from $\sigma$, $\varphi$ and $\mu$ is to assume that $h(\cdot)$ is bounded.  The simplest choice is
\begin{align*}
h(\varphi) = \max \left ( 0, \min \left ( \frac{1}{2}(\varphi + 1) , 1 \right ) \right ),
\end{align*}
which ensures $h(-1) = 0$ and $h(1) = 1$ as requested.  By considering the bounded functions $h(\cdot)$, we can control the source terms $\mu h(\varphi)(\lambda_{a} - \lambda_{p} \sigma) + N_{,\sigma} \lambda_{c} \sigma h(\varphi)$ in \eqref{Intro:energyid}, and thus applications of H\"{o}lder's inequality and Young's inequality will lead to (see \eqref{apriori:totaltimeenergymean:sim} below)
\begin{equation}\label{Intro:energy:apriori}
\begin{aligned}
& \frac{\dd}{\dt} \int_{\Omega} \left [ A \Psi(\varphi) + \frac{B}{2} \abs{\nabla \varphi}^{2} + \frac{\chi_{\sigma}}{2} \abs{\sigma}^{2} + \chi_{\varphi} \sigma (1-\varphi) \right ] \dx \\
& \quad +  k_{1} \left ( \norm{\nabla \mu}_{L^{2}(\Omega)}^{2} + \norm{\nabla N_{,\sigma}}_{L^{2}(\Omega)}^{2} + \norm{\sigma}_{L^{2}(\Gamma)}^{2} \right ) \\
& \quad  - k_{2} \norm{\sigma}_{L^{2}(\Omega)}^{2} - k_{3} \norm{\varphi}_{L^{2}(\Omega)}^{2} - k_{4} \norm{\nabla \varphi}_{L^{2}(\Omega)}^{2} \leq C,
\end{aligned}
\end{equation}
for some positive constants $k_{1}, k_{2}, k_{3}, k_{4}$ and $C$.  The sign indefiniteness of the term $\chi_{\varphi} \sigma (1-\varphi)$ means that we have to first integrate \eqref{Intro:energy:apriori} in time and then estimate with H\"{o}lder's inequality and Young's inequality.  Thus, we obtain
\begin{equation}
\label{Intro:energy:apriori:integrated}
\begin{aligned}
&  A \norm{\Psi(\varphi)}_{L^{1}(\Omega)} + \frac{B}{2} \norm{\nabla \varphi}_{L^{2}(\Omega)}^{2} + k_{5} \norm{\sigma}_{L^{2}(\Omega)}^{2} - k_{6} \norm{\varphi}_{L^{2}(\Omega)}^{2} \\
& \quad +   k_{1} \int_{0}^{T} \left (\norm{\nabla \mu}_{L^{2}(\Omega)}^{2} + \norm{\nabla N_{,\sigma}}_{L^{2}(\Omega)}^{2} + \norm{\sigma}_{L^{2}(\Gamma)}^{2} \right ) \dt \\
& \quad - k_{2} \norm{\sigma}_{L^{2}(0,T;L^{2}(\Omega))}^{2} - k_{3} \norm{\varphi}_{L^{2}(0,T;L^{2}(\Omega))}^{2} - k_{4} \norm{\nabla \varphi}_{L^{2}(0,T;L^{2}(\Omega))}^{2} \leq C,
\end{aligned}
\end{equation}
for some positive constants $k_{5}$, $k_{6}$ and $C$.  A structural assumption \eqref{assump:Psi} on the potential $\Psi$  will allow us to control $\norm{\varphi}_{L^{2}(\Omega)}^{2}$ with $\norm{\Psi}_{L^{1}(\Omega)}$ (see \eqref{varphiL2PsiL1} below).  This will lead to
\begin{equation}
\label{Intro:energy:apriori:Gronwall}
\begin{aligned}
& (A - k_{7}) \norm{\Psi(\varphi)}_{L^{1}(\Omega)} + \frac{B}{2} \norm{\nabla \varphi}_{L^{2}(\Omega)}^{2} + k_{5} \norm{\sigma}_{L^{2}(\Omega)}^{2} \\
& \quad +  k_{1} \int_{0}^{T} \left ( \norm{\nabla \mu}_{L^{2}(\Omega)}^{2} + \norm{\nabla N_{,\sigma}}_{L^{2}(\Omega)}^{2} + \norm{\sigma}_{L^{2}(\Gamma)}^{2} \right ) \dt \\
& \quad - k_{2} \norm{\sigma}_{L^{2}(0,T;L^{2}(\Omega))}^{2} - k_{8} \norm{\Psi(\varphi)}_{L^{1}(0,T;L^{1}(\Omega))} - k_{4} \norm{\nabla \varphi}_{L^{2}(0,T;L^{2}(\Omega))}^{2} \leq C,
\end{aligned}
\end{equation}
for some positive constants $k_{7}$, $k_{8}$ and $C$.  To apply the integral version of Gronwall's inequality, we have to assume that the constant $A$ satisfies $A > k_{7}$.  This is needed in order to derive the usual a priori bounds for $\varphi$ and $\mu$ in Cahn--Hilliard systems with source terms.  However, we point out that the constant $A$ is often chosen to be $A := \frac{\gamma}{\eps}$, where $\gamma > 0$ denotes the surface tension and $\eps > 0$ is a small parameter related to the interfacial thickness.  For sufficiently small values of $\eps$ or sufficiently large surface tension $\gamma$, we see that $A > k_{7}$ will be satisfied, and thus it is not an unreasonable constraint.

Let us consider the nutrient equation \eqref{Intro:sigma} with the specific choice of fluxes \eqref{activetransportchoice}, leading to
\begin{align*}
\pd_{t} \sigma = \div (D(\varphi) \nabla \sigma) - \eta \div (D(\varphi) \nabla \varphi) - \lambda_{c} \sigma h(\varphi).
\end{align*} 
Performing a non-dimensionalisation leads to the following non-dimensionalised nutrient equation (here we reuse the same notation to denote the non-dimensionalised variables)
\begin{align}\label{nondim:sigma}
\kappa \pd_{t} \sigma = \Laplace \sigma - \theta \Laplace \varphi - \alpha \sigma h(\varphi),
\end{align}
where $\kappa > 0$ represents the ratio between the nutrient diffusion timescale and the tumour doubling timescale, $\theta > 0$ represents the ratio between the nutrient diffusion time-scale and the active transport timescale, and $\alpha > 0$ represents the ratio between the nutrient diffusion timescale and the nutrient consumption timescale.  

In practice, experimental values indicate that $\kappa \ll 1$ (see for example \cite[Section 4.3.2]{incoll:ByrneCancerbook}) and we assume that the timescales of nutrient active transport and nutrient consumption are of the same order as the timescale of nutrient diffusion, i.e., $\theta \sim \mathcal{O}(1)$, $\alpha \sim \mathcal{O}(1)$.  This leads to the following quasi-static model,
\begin{subequations}\label{Intro:CHN:quasi}
\begin{alignat}{2}
\pd_{t}\varphi  &= \div (m(\varphi) \nabla \mu) + (\lambda_{p} \sigma - \lambda_{a}) h(\varphi) && \text{ in } \Omega \times (0,T), \label{Intro:quasi:varphi} \\
\mu & = A\Psi'(\varphi) - B \Laplace \varphi - \chi_{\varphi} \sigma && \text{ in } \Omega \times (0,T), \label{Intro:quasi:mu} \\
0 & = \div (D(\varphi) \nabla \sigma) - \eta \div(D(\varphi) \nabla \varphi) -\lambda_{c} \sigma h(\varphi) && \text{ in } \Omega \times (0,T),   \label{Intro:quasi:sigma} \\
0 & = \nabla \varphi \cdot \nu = \nabla \mu \cdot \nu && \text{ on } \Gamma \times (0,T), \\
D(\varphi) \nabla \sigma \cdot \nu & = K (\sigma_{\infty} - \sigma) && \text{ on } \Gamma \times (0,T). \label{Intro:quasi:sigma:boundary}
\end{alignat}
\end{subequations}
Note that the loss of the time derivative $\pd_{t} \sigma$ implies that an energy identity for \eqref{Intro:CHN:quasi} cannot be derived in a similar fashion to \eqref{Intro:energyid}.  However, if we test \eqref{Intro:quasi:mu} with $\pd_{t}\varphi$, \eqref{Intro:quasi:varphi} with $\chi_{\varphi} \sigma + \mu$, \eqref{Intro:quasi:sigma} with $\sigma$ and add the resulting equations, we formally obtain
\begin{equation}\label{Intro:Quasi:EnergyID}
\begin{aligned}
& \frac{\dd}{\dt} \int_{\Omega} \left [ A \Psi(\varphi) + \frac{B}{2} \abs{\nabla \varphi}^{2} \right ] \dx \\
& \quad + \int_{\Omega} m(\varphi) \abs{\nabla \mu}^{2} + D(\varphi) \abs{\nabla \sigma}^{2} + \lambda_{c} h(\varphi) \abs{\sigma}^{2} \dx + \int_{\Gamma} K \abs{\sigma}^{2} \dHaus \\
& \quad = \int_{\Omega} - m(\varphi) \chi_{\varphi} \nabla \mu \cdot \nabla \sigma  + D(\varphi) \eta \nabla \varphi \cdot \nabla \sigma \dx + \int_{\Omega} (\lambda_{p} \sigma - \lambda_{a}) h(\varphi)(\chi_{\varphi} \sigma + \mu) \dx  \\
& \quad + \int_{\Gamma}  K  \sigma \sigma_{\infty} \dHaus.
\end{aligned}
\end{equation}
Here, we point out that there are no terms with indefinite sign under the time derivative, and so we expect that there will not be a restriction on the constant $A$ as in the model \eqref{Intro:CHNutrient}.  In principle, we can also perform the same testing procedure to \eqref{Intro:varphi}, \eqref{Intro:mu}, and \eqref{nondim:sigma} to obtain a similar identity to \eqref{Intro:Quasi:EnergyID} with an additional term $\frac{\dd}{\dt} \frac{\kappa}{2} \norm{\sigma}_{L^{2}(\Omega)}^{2}$ on the left-hand side.  However, the a priori estimates obtain from a Gronwall argument will not be uniform in $\kappa$, which is due to the fact that the source terms involving $\sigma$ on the right-hand side cannot be bounded any longer with the help of $\frac{\kappa}{2} \norm{\sigma}_{L^{2}(\Omega)}^{2}$ on the left-hand side.  

Thus in this work, we cannot realize \eqref{Intro:CHN:quasi} as a limit system from \eqref{Intro:varphi}, \eqref{Intro:mu}, and \eqref{nondim:sigma} as $\kappa \to 0$, and the well-posedness of \eqref{Intro:CHN:quasi} will be proved separately.  However, if we supplement \eqref{Intro:varphi}, \eqref{Intro:mu}, and \eqref{nondim:sigma} with Dirichlet boundary conditions, then we can rigorously establish the quasi-static system \eqref{Intro:CHN:quasi} as a limit system of \eqref{Intro:varphi}, \eqref{Intro:mu}, and \eqref{nondim:sigma} as $\kappa \to 0$.  For more details, we refer to \cite{article:GarckeLamDirichlet}.

We now compare \eqref{Intro:CHNutrient} with the other models for tumour growth studied in the literature.  In \cite{article:HawkinsZeeOden12}, the authors derived the following model,
\begin{subequations}\label{ModelHawkins}
\begin{align}
\pd_{t}\varphi  &= \div (m(\varphi) \nabla \mu) + P(\varphi)(\chi_{\sigma} \sigma + \chi_{\varphi}(1- \varphi)  - \mu), \label{Model:Hawkins:varphi} \\
\mu & = A\Psi'(\varphi) - B \Laplace \varphi - \chi_{\varphi} \sigma, \label{Model:Hawkins:mu} \\
\pd_{t}\sigma  & = \div (n(\varphi) (\chi_{\sigma} \nabla \sigma - \chi_{\varphi} \nabla \varphi)) -P(\varphi)(\chi_{\sigma} \sigma + \chi_{\varphi}(1- \varphi) - \mu), \label{Model:Hawkins:sigma}
\end{align}
\end{subequations}
where we see that the chemical potentials $N_{,\sigma}$ and $\mu$ enter as source terms in \eqref{Model:Hawkins:varphi} and \eqref{Model:Hawkins:sigma}, and $P(\varphi)$ is a non-negative function. Subsequently, if we consider
\begin{align*}
\chi_{\sigma} = 1, \quad \chi_{\varphi} = 0, \quad n(\varphi) = m(\varphi) = 1
\end{align*} 
in \eqref{ModelHawkins}, then we obtain
\begin{subequations}\label{ModelRocca}
\begin{align}
\pd_{t}\varphi &= \Laplace \mu + P(\varphi)(\sigma - \mu), \label{Model:Rocca:varphi} \\
\mu & = A\Psi'(\varphi) - B \Laplace \varphi, \label{Model:Rocca:mu} \\
\pd_{t}\sigma & = \Laplace \sigma -P(\varphi)(\sigma - \mu). \label{Model:Rocca:sigma}
\end{align}
\end{subequations}
Furnishing \eqref{ModelRocca} with homogeneous Neumann boundary conditions, the well-posedness of the system and the existence of the global attractor have been proved in \cite{article:FrigeriGrasselliRocca15} for large classes of nonlinearities $\Psi$ and $P$.

The corresponding viscosity regularised version of \eqref{ModelRocca} (where there is an extra $\alpha \pd_{t}\mu$ term on the left-hand side of \eqref{Model:Rocca:varphi} and an extra $\alpha \pd_{t}\varphi$ term on the right-hand side of \eqref{Model:Rocca:mu} for positive constant $\alpha$) has been studied in \cite{article:ColliGilardiHilhorst15}, where well-posedness is proved for a general class of potentials $\Psi$, and for a Lipschitz and globally bounded $P$.  The asymptotic behaviour as $\alpha \to 0$ is shown under more restrictions on $\Psi$ (polynomial growth of order 4) and the authors proved that a sequence of weak solutions to the viscosity regularised system converges to the weak solution of \eqref{ModelRocca}.  Further investigation in obtaining convergence rates with singular potentials have been initiated in \cite{article:ColliGilardiRoccaSprekelsVV,article:ColliGilardiRoccaSprekelsAA}, and the corresponding sharp interface limit is obtained via a formally matched asymptotic analysis performed in \cite{article:Kampmann}.

For \eqref{ModelRocca}, there is a natural Lyapunov-type energy equality given as
\begin{equation}\label{Rocca:energy}
\begin{aligned}
& \frac{\dd}{\dt} \int_{\Omega} \left [ A \Psi(\varphi) + \frac{B}{2} \abs{\nabla \varphi}^{2} + \frac{1}{2} \abs{\sigma}^{2} \right ] \dx \\
& \quad +  \norm{\nabla \mu}_{L^{2}(\Omega)}^{2} + \norm{\nabla \sigma}_{L^{2}(\Omega)}^{2} + \int_{\Omega} P(\varphi)(\sigma - \mu)^{2} \dx = 0.
\end{aligned}
\end{equation}
Since all the terms are non-negative, the standard a priori estimates can be obtained even in the case where $\Psi$ has polynomial growth of order $6$ in three dimensions.  In contrast, for \eqref{Intro:CHNutrient} we have to assume that the derivative $\Psi'$ has linear growth, thus restricting our class of potentials to those with at most quadratic growth (see Section \ref{sec:discussion} below).

The quasi-static model \eqref{Intro:CHN:quasi} bears the most resemblance to \cite[Equations (68)-(70)]{article:CristiniLiLowengrubWise09} when the active transport is neglected (i.e., $\eta = 0$).  We note that the focus of study seems to be the linear stability of radial solutions to the resulting sharp interface limit  when we set $A = \frac{1}{\eps}$ and $B = \eps$, and send $\eps \to 0$.  To the best of our knowledge, there are no results concerning the well-posedness of \eqref{Intro:CHN:quasi}.

We also mention another class of models that describes tumour growth using a Cahn--Hilliard--Darcy system,
\begin{subequations}
\label{Model:LowengrubTitiZhao}
\begin{align}
\div \bm{v} & = \mathcal{S}, \\
\bm{v} &= -M(\nabla p + \mu \nabla \varphi), \label{LTZ:velo} \\
\pd_{t} \varphi + \div (\bm{v} \varphi) & = \nabla \cdot (m(\varphi) \nabla \mu) + \mathcal{S}, \\
\mu & = A\Psi'(\varphi) - B \Laplace \varphi,
\end{align}
\end{subequations}
where $\bm{v}$ denote a mixture velocity, $p$ denotes the pressure, $M$ is the permeability, and $\mathcal{S}$ denotes a mass exchange term.  For the case where $\mathcal{S} = 0$ and $M = 1$, the existence of strong solutions in 2D and 3D have been studied in \cite{article:LowengrubTitiZhao13}.  The global existence of weak solutions in two and three dimensions via the convergence of a fully discrete and energy stable implicit finite element scheme is established in \cite{article:FengWise12}, and uniqueness of weak solutions can be shown if additional regularity assumptions on the solutions are imposed.  For the case where $\mathcal{S} \neq 0$ is prescribed and $M = 1$, existence of global weak solutions in 2D and 3D, and unique local strong solutions in 2D can be found in \cite{article:JiangWuZheng14}.  A related system, known as the Cahn--Hilliard--Brinkman system, where an additional viscosity term is added to the left-hand side of the velocity equation \eqref{LTZ:velo} and the mass exchange $\mathcal{S}$ is set to zero, has been the subject of study in \cite{preprint:BosiaContiGrasselli14}.  Meanwhile, in the case $\mathcal{S} = 0$ and $M$ is a function depending on $\varphi$, the system \eqref{Model:LowengrubTitiZhao} is also referred to as the Hele--Shaw--Cahn--Hilliard model (see \cite{article:LeeLowengrubGoodman01,article:LeeLowengrubGoodman01:Part2}).  In this setting, $M$ is the reciprocal of the viscosity of the fluid mixture, and we refer to  \cite{article:WangZhang} concerning strong well-posedness globally in time for two dimensions and locally in time for three dimensions when $\Omega$ is the $d$-dimensional torus.  Long-time behaviour of solutions to the Hele--Shaw--Cahn--Hilliard model is studied in \cite{article:WangWu}.  

The structure of this paper is as follows.  In Section \ref{sec:mainresults}, we state the assumptions and the well-posedness results for \eqref{Intro:CHNutrient} and \eqref{Intro:CHN:quasi}.  In Section \ref{sec:apriori} we derive some useful estimates, and in Section \ref{sec:Existence}, we prove the existence of weak solutions to \eqref{Intro:CHNutrient} via a Galerkin procedure.  Continuous dependence on initial and boundary data for \eqref{Intro:CHNutrient} is shown in Section \ref{sec:Uniqueness}.  In Section \ref{sec:quasi}, we outline the proof of well-posedness for \eqref{Intro:CHN:quasi}, and in Section \ref{sec:discussion} we discuss the issue of the growth assumptions for the potential.

\section{Main results}\label{sec:mainresults}

\paragraph{Notation and useful preliminaries.}
For convenience, we will often use the notation $L^{p} := L^{p}(\Omega)$ and $W^{k,p} := W^{k,p}(\Omega)$ for any $p \in [1,\infty]$, $k > 0$ to denote the standard Lebesgue spaces and Sobolev spaces equipped with the norms $\norm{\cdot}_{L^{p}}$ and $\norm{\cdot}_{W^{k,p}}$.  Moreover, the dual space of a Banach space $X$ will be denoted by $X^{*}$.  In the case $p = 2$, we use $H^{k} := W^{k,2}$ with the norm $\norm{\cdot}_{H^{k}}$.

For any $d \in \N$, let $\Omega \subset \R^{d}$ denote a bounded domain with Lipschitz boundary $\Gamma$, and let $T > 0$.  We recall the Poincar\'{e} inequalities (see for instance \cite[Equations (1.35), (1.37a) and (1.37c)]{book:Temam}): There exists a positive constant $C_{\mathrm{P}}$, depending only on $\Omega$ and the dimension $d$, such that for all $f \in H^{1}$,
\begin{align}
\bignorm{f - \overline{f}}_{L^{2}} & \leq C_{\mathrm{P}} \norm{\nabla f}_{L^{2}}, \label{regular:Poincare} \\
\norm{f}_{L^{2}} & \leq C_{\mathrm{P}} \left (\norm{\nabla f}_{L^{2}} + \norm{f}_{L^{2}(\Gamma)} \right ), \label{boundary:Poincare}
\end{align}
where $\mean{f} := \frac{1}{\abs{\Omega}} \int_{\Omega} f \dx$ denotes the mean of $f$.

\begin{assump}\label{assump:main}
\
\begin{enumerate}[label=$(\mathrm{A \arabic*})$, ref = $(\mathrm{A \arabic*})$] 
\item $\lambda_{p}$, $\lambda_{a}$, $\lambda_{c}$ and $\chi_{\varphi}$ are fixed non-negative constants, while $\chi_{\sigma}$, $A$, $B$ and $K$ are fixed positive constants.
\item The initial and boundary data satisfy 
\begin{align*}
\varphi_{0} \in H^{1}, \quad \sigma_{0} \in L^{2}, \quad \sigma_{\infty} \in L^{2}(0,T;L^{2}(\Gamma)).
\end{align*}
\item The functions $m$, $n$, $h$ and $D$ belong to the space $C^{0}(\R)$, and there exist positive constants $h_{\infty}$, $m_{0}$, $m_{1}$, $D_{0}$, $D_{1}$, $n_{0}$ and $n_{1}$, such that for all $t \in \R$,
\begin{align}
\label{assump:mn}
m_{0} \leq m(t) \leq m_{1}, \quad n_{0} \leq n(t) \leq n_{1}, \quad D_{0} \leq D(t) \leq D_{1}, \quad 0 \leq h(t) \leq h_{\infty}.
\end{align}
\item The potential $\Psi \in C^{1,1}(\R)$ is non-negative, continuously differentiable, with globally Lipschitz derivative and satisfies 
\begin{align}\label{assump:Psi}
\Psi(t) \geq R_{1} \abs{t}^{2} - R_{2}  , \quad \abs{\Psi'(t)} \leq R_{3} (1 + \abs{t}),
\end{align}
for positive constants $R_{2}$, $R_{3}$ and a positive constant $R_{1}$ such that
\begin{align}\label{assump:constantsrelations}
 A > \frac{2 \chi_{\varphi}^{2}}{\chi_{\sigma} R_{1}}.
\end{align}
\end{enumerate}
\end{assump}

\begin{defn}\label{defn:weaksolution}
We call a triplet of functions $(\varphi, \mu, \sigma)$ a weak solution to \eqref{Intro:CHNutrient} if
\begin{align*}
\sigma, \varphi \in H^{1}(0,T;(H^{1})^{*}) \cap L^{2}(0,T;H^{1}), \quad \mu  \in L^{2}(0,T;H^{1}),
\end{align*}
with $\varphi(0) = \varphi_{0}$, $\sigma(0) = \sigma_{0}$ and satisfy for $\zeta, \phi, \xi \in H^{1}$ and a.e. $t \in (0,T)$, 
\begin{subequations}\label{CHNutrient:truncated:weakform}
\begin{align}
\inner{\pd_{t} \varphi}{\zeta} & = \int_{\Omega} - m(\varphi) \nabla \mu \cdot \nabla \zeta + (\lambda_{p} \sigma - \lambda_{a}) h(\varphi) \zeta \dx, \label{Truncated:varphi:weak} \\
\int_{\Omega} \mu \phi \dx & = \int_{\Omega} A \Psi'(\varphi) \phi + B \nabla \varphi \cdot \nabla \phi - \chi_{\varphi} \sigma \phi \dx, \label{Truncated:mu:weak} \\
 \inner{\pd_{t} \sigma}{\xi} & = \int_{\Omega} - n(\varphi) (\chi_{\sigma} \nabla \sigma - \chi_{\varphi} \nabla \varphi) \cdot \nabla \xi - \lambda_{c} \sigma h(\varphi) \xi \dx \label{Truncated:sigma:weak}\\
\notag & + \int_{\Gamma} \xi K (\sigma_{\infty} - \sigma) \dHaus, 
\end{align}
\end{subequations}
where $\inner{\cdot}{\cdot}$ denotes the duality pairing between $H^{1}$ and its dual $(H^{1})^{*}$.
\end{defn}

\begin{thm}[Existence of global weak solutions]\label{thm:existence}
Let $\Omega \subset \R^{d}$ be a bounded domain with Lipschitz boundary $\Gamma$ and let $T > 0$.  Suppose Assumption \ref{assump:main} is satisfied.  Then, there exists a triplet of functions $(\varphi, \mu, \sigma)$ such that
\begin{align*}
\varphi & \in L^{\infty}(0,T;H^{1}) \cap H^{1}(0,T;(H^{1})^{*}), \quad \mu \in L^{2}(0,T;H^{1}), \\
\sigma & \in L^{2}(0,T;H^{1}) \cap L^{\infty}(0,T;L^{2}) \cap H^{1}(0,T;(H^{1})^{*}),
\end{align*}
and is a weak solution of \eqref{Intro:CHNutrient} in the sense of Definition \ref{defn:weaksolution}.
\end{thm}
The embedding of $L^{2}(0,T;H^{1}) \cap H^{1}(0,T;(H^{1})^{*})$ into $C([0,T]; L^{2})$ guarantees that the initial data are meaningful.  We point out that the assumption \eqref{assump:constantsrelations} arises from using Young's inequality to estimate the term $\chi_{\varphi} \sigma (1-\varphi)$ in \eqref{Intro:energyid}, and is by no means an optimal assumption.  See Remark \ref{remark:constantrelations} for more details.  In addition, Theorem \ref{thm:existence} gives the existence of weak solutions in any dimension.  This is thanks to the fact that $\Psi'$ has linear growth (see $\eqref{assump:Psi}_{2}$).

Next, we show continuous dependence on initial and boundary data and uniqueness of weak solutions under additional assumptions on the interpolation function $h(\cdot)$ and the mobilities $m(\cdot)$ and $n(\cdot)$.
\begin{thm}[Continuous dependence and uniqueness]\label{thm:ctsdep}
Let $d \leq 4$.  Suppose $h(\cdot) \in C^{0,1}(\R)$,  $m(\cdot)$ and $n(\cdot)$ are constant mobilities $($without loss of generality we set $m(\cdot) = n(\cdot) = 1)$.  For $i = 1,2$, let
\begin{align*}
\varphi_{i} & \in L^{\infty}(0,T;H^{1}) \cap H^{1}(0,T;(H^{1})^{*}), \quad \mu_{i} \in L^{2}(0,T;H^{1}), \\
\sigma_{i} & \in L^{2}(0,T;H^{1}) \cap L^{\infty}(0,T;L^{2}) \cap H^{1}(0,T;(H^{1})^{*})
\end{align*}
denote two weak solutions of \eqref{Intro:CHNutrient} satisfying \eqref{CHNutrient:truncated:weakform} with corresponding initial data $\varphi_{i}(0) = \varphi_{0,i} \in H^{1}$, $\sigma_{i}(0) = \sigma_{0,i} \in L^{2}$, and boundary data $\sigma_{\infty,i} \in L^{2}(0,T;L^{2}(\Gamma))$.  Then,
\begin{align*}
 \sup_{s \in [0,T]} & \left (\norm{\sigma_{1}(s) - \sigma_{2}(s)}_{L^{2}}^{2} + \norm{\varphi_{1}(s) - \varphi_{2}(s)}_{L^{2}}^{2} \right ) \\
& +  \norm{\mu_{1} - \mu_{2}}_{L^{2}(0,T;L^{2})}^{2} + \norm{\nabla (\sigma_{1} - \sigma_{2})}_{L^{2}(0,T;L^{2})}^{2} \\
& + \norm{\sigma_{1} - \sigma_{2}}_{L^{2}(0,T;L^{2}(\Gamma))}^{2} + \norm{\nabla(\varphi_{1} - \varphi_{2})}_{L^{2}(0,T;L^{2})}^{2} \\
& \leq  C \left (\norm{\sigma_{0,1} - \sigma_{0,2}}_{L^{2}}^{2} + \norm{\varphi_{0,1} - \varphi_{0,2}}_{L^{2}}^{2} + \norm{\sigma_{\infty,1} - \sigma_{\infty,2}}_{L^{2}(0,T;L^{2}(\Gamma))}^{2} \right ),
\end{align*}
where the constant $C$ depends on $\norm{\sigma_{i}}_{L^{\infty}(0,T;L^{2})}$, $T$, $K$, $h_{\infty}$, $\Omega$, $d$, $A$, $B$, $\lambda_{p}$, $\lambda_{c}$, $\lambda_{a}$, $\chi_{\varphi}$, $\chi_{\sigma}$, and $\mathrm{L}_{h}$, $\mathrm{L}_{\Psi'}$ which denote the Lipschitz constants of $h$ and $\Psi'$, respectively.
\end{thm}

We point out that Theorem \ref{thm:ctsdep} provides continuous dependence for the difference of the chemical potentials $\norm{\mu_{1} - \mu_{2}}_{L^{2}(\Omega \times (0,T))}$ and also with a stronger norm $\norm{\varphi_{1}(t) - \varphi_{2}(t)}_{L^{\infty}(0,T;L^{2})}$ for the difference of the order parameters.  This is in contrast with the classical norm $\norm{\varphi_{1}(t) - \varphi_{2}(t)}_{L^{\infty}(0,T;(H^{1})^{*})}$ one obtains for the Cahn--Hilliard equation, compare \cite[Theorem 2]{article:FrigeriGrasselliRocca15}.

We will now consider the quasi-static system \eqref{Intro:CHN:quasi}.

\begin{defn}\label{defn:weaksolutionquasi}
We call a triplet of functions $(\varphi, \mu, \sigma)$ a weak solution to \eqref{Intro:CHN:quasi} if
\begin{align*}
\sigma, \mu \in L^{2}(0,T;H^{1}), \quad \varphi \in H^{1}(0,T;(H^{1})^{*}) \cap L^{2}(0,T;H^{1}), 
\end{align*}
with $\varphi(0) = \varphi_{0}$ and satisfy for $\zeta, \lambda, \xi \in H^{1}$ and a.e. $t \in (0,T)$, 
\begin{subequations}\label{CHNutrient:quasi:weakform}
\begin{align}
\inner{\pd_{t} \varphi}{\zeta} & = \int_{\Omega} - m(\varphi) \nabla \mu \cdot \nabla \zeta + (\lambda_{p} \sigma - \lambda_{a}) h(\varphi) \zeta  \dx, \label{quasi:varphi:weak} \\
\int_{\Omega} \mu \lambda \dx & = \int_{\Omega} A \Psi'(\varphi) \lambda + B \nabla \varphi \cdot \nabla \lambda - \chi_{\varphi} \sigma \lambda \dx, \label{quasi:mu:weak} \\
\int_{\Gamma} \xi K (\sigma_{\infty} - \sigma) \dHaus & = \int_{\Omega} D(\varphi) (\nabla \sigma - \eta \nabla \varphi) \cdot \nabla \xi + \lambda_{c} \sigma h(\varphi) \xi \dx. \label{quasi:sigma:weak}
\end{align}
\end{subequations}
\end{defn}

\begin{thm}[Existence and regularity of global weak solutions]\label{thm:quasi:exist}
Let $\Omega \subset \R^{d}$ be a bounded domain with Lipschitz boundary $\Gamma$ and let $T > 0$.  Suppose Assumption \ref{assump:main} is satisfied, and let $A$ be a positive constant which need not satisfy \eqref{assump:constantsrelations}.  Then, there exists a triplet of functions $(\varphi, \mu, \sigma)$ such that
\begin{align*}
\sigma, \mu \in L^{2}(0,T;H^{1}), \quad \varphi  \in L^{\infty}(0,T;H^{1}) \cap H^{1}(0,T;(H^{1})^{*}), 
\end{align*}
and is a weak solution of \eqref{Intro:CHN:quasi} in the sense of Definition \ref{defn:weaksolutionquasi}.  Furthermore, if $\sigma_{\infty} \in L^{\infty}(0,T;L^{2}(\Gamma))$, then
\begin{align*}
\sigma \in L^{\infty}(0,T;H^{1}).
\end{align*}
\end{thm}
In Section \ref{sec:quasi} we derive the a priori estimates and deduce the existence of approximate solutions on the Galerkin level.  The proof of Theorem \ref{thm:quasi:exist} then follows from standard compactness results.  In Section \ref{sec:quasi:ctsdep}, we show the continuous dependence on initial and boundary data and uniqueness under additional assumptions.

\begin{thm}[Continuous dependence and uniqueness]\label{thm:quasi:ctsdep}
Let $d \leq 4$.  Suppose $h(\cdot) \in C^{0,1}(\R)$,  $m$ and $D$ are constant mobilities $($without loss of generality we set $m = 1)$.  For $i = 1,2$, let
\begin{align*}
\varphi_{i} \in L^{\infty}(0,T;H^{1}) \cap H^{1}(0,T;(H^{1})^{*}), \quad \mu_{i} \in L^{2}(0,T;H^{1}), \quad \sigma_{i}  \in L^{\infty}(0,T;H^{1})
\end{align*}
denote two weak solutions of \eqref{Intro:CHN:quasi} satisfying \eqref{CHNutrient:quasi:weakform} with corresponding initial data $\varphi_{i}(0) = \varphi_{0,i} \in H^{1}$ and boundary data $\sigma_{\infty,i} \in L^{\infty}(0,T;L^{2}(\Gamma))$.  Then,
\begin{align*}
\sup_{s \in [0,T]}  & \norm{\varphi_{1}(s) - \varphi_{2}(s)}_{L^{2}}^{2} + \norm{\mu_{1} - \mu_{2}}_{L^{2}(0,T;L^{2})}^{2}  + \norm{\nabla(\varphi_{1} - \varphi_{2})}_{L^{2}(0,T;L^{2})}^{2}\\
& + \norm{\nabla (\sigma_{1} - \sigma_{2})}_{L^{2}(0,T;L^{2})}^{2} + \norm{\sigma_{1} - \sigma_{2}}_{L^{2}(0,T;L^{2}(\Gamma))}^{2} \\
& \leq  C \left ( \norm{\varphi_{0,1} - \varphi_{0,2}}_{L^{2}}^{2} + \norm{\sigma_{\infty,1} - \sigma_{\infty,2}}_{L^{2}(0,T;L^{2}(\Gamma))}^{2} \right ),
\end{align*}
where the constant $C$ depends on $\norm{\sigma_{i}}_{L^{\infty}(0,T;H^{1})}$, $K$, $\Omega$, $A$, $B$, $\mathrm{L}_{h}$, $\mathrm{L}_{\Psi'}$, $\lambda_{p}$, $\lambda_{c}$, $\lambda_{a}$, $\chi_{\varphi}$, and $T$.
\end{thm}

\section{Useful estimates}\label{sec:apriori}
We will use a modified version of Gronwall's inequality in integral form.
\begin{lemma}\label{lem:GronwallInt}
Let $\alpha, \beta, u$ and $v$ be real-valued functions defined on $I:=[0,T]$.  Assume that $\alpha$ is integrable, $\beta$ is non-negative and continuous, $u$ is continuous, $v$ is non-negative and integrable.  Suppose $u$ and $v$ satisfy the integral inequality
\begin{align}\label{Gronwall:int:ineq}
u(s) + \int_{0}^{s} v(t) \dt \leq \alpha(s) + \int_{0}^{s} \beta(t) u(t) \dt \quad \forall s \in I.
\end{align}
Then,
\begin{align}\label{Gronwall}
u(s) + \int_{0}^{s} v(t) \dt \leq \alpha(s) + \int_{0}^{s} \alpha(t) \beta(t) \exp \left ( \int_{t}^{s} \beta(r) \dr \right ) \dt.
\end{align}
\end{lemma}
This differs from the usual Gronwall's inequality in integral form by an extra term $\int_{0}^{s} v(t) \dt$ on the left-hand side.
\begin{proof}
Let 
\begin{align*}
w(s) := u(s) + \int_{0}^{s} v(t) \dt.
\end{align*}
Then, by \eqref{Gronwall:int:ineq} and the non-negativity of $\beta$ and $v$, 
\begin{align*}
w(s) \leq \alpha(s) + \int_{0}^{s} \beta(t) w(t) \dt.
\end{align*}
Applying the standard Gronwall's inequality in integral form yields the required result.
\end{proof}

Below we will derive the first a priori estimate for sufficiently smooth solutions to \eqref{Intro:CHNutrient}, in particular this will hold for the Galerkin approximations in Section \ref{sec:Galerkinapprox}.  We choose to present this estimate here due to the length of the derivation.

\begin{lemma}\label{lem:MainAprioriEst:1}
Suppose Assumption \ref{assump:main} is satisfied.  Let $(\varphi, \mu, \sigma)$ be a triplet of functions satisfying \eqref{CHNutrient:truncated:weakform} with $\varphi(0) = \varphi_{0}$ and $\sigma(0) = \sigma_{0}$, and $\varphi, \sigma \in C^{1}([0,T];H^{1})$, $\mu \in C^{0}([0,T];H^{1})$.  Then, there exists a positive constant $\overline{C}$ depending on $T$, $\Omega$, $\Gamma$, $d$, $R_{1}$, $R_{2}$, $R_{3}$, the parameters $\lambda_{p}$, $\lambda_{a}$, $\lambda_{c}$, $\chi_{\sigma}$, $\chi_{\varphi}$, $h_{\infty}$, $m_{0}$, $n_{0}$, $A$, $B$, $K$, the initial-boundary data $\norm{\sigma_{\infty}}_{L^{2}(0,T;L^{2}(\Gamma))}$, $\norm{\varphi(0)}_{H^{1}}$ and $\norm{\sigma(0)}_{L^{2}}$, such that for all $s \in (0,T]$,
\begin{equation}\label{apriori:main}
\begin{aligned}
& \norm{\Psi(\varphi(s))}_{L^{1}} + \norm{\varphi(s)}_{H^{1}}^{2} + \norm{\sigma(s)}_{L^{2}}^{2}  \\
& \quad +  \norm{\nabla \mu}_{L^{2}(0,s;L^{2})}^{2}+ \norm{ \nabla \sigma }_{L^{2}(0,s;L^{2})}^{2} + \norm{\sigma}_{L^{2}(0,s;L^{2}(\Gamma))}^{2} \leq \overline{C}.
\end{aligned}
\end{equation}
\end{lemma}

\begin{proof}
Let
\begin{align}\label{defn:initialenergy}
c_{0} := \int_{\Omega} \left [ A \Psi(\varphi_{0}) + \frac{B}{2} \abs{\nabla \varphi_{0}}^{2} + \frac{\chi_{\sigma}}{2} \abs{\sigma_{0}}^{2} + \chi_{\varphi} \sigma_{0}(1-\varphi_{0}) \right ] \dx
\end{align}
denote the initial energy.  Then, by the assumption on the $\varphi_{0}$ and $\sigma_{0}$, H\"{o}lder's inequality and Young's inequality we see that $c_{0}$ is bounded.

Substituting $\zeta = \mu$, $\phi = \pd_{t}\varphi$, and $\xi = \chi_{\sigma} \sigma + \chi_{\varphi}(1-\varphi) = N_{,\sigma}$ into \eqref{CHNutrient:truncated:weakform} and adding the resulting equations together, we obtain
\begin{equation}
\label{apriori:totalenergymean}
\begin{aligned}
&  \frac{\dd}{\dt} \int_{\Omega}  \left [ A \Psi(\varphi) + \frac{B}{2} \abs{\nabla \varphi}^{2} + \frac{\chi_{\sigma}}{2} \abs{\sigma}^{2} + \chi_{\varphi} \sigma(1-\varphi) \right ] \dx \\
& \quad + \int_{\Omega} m(\varphi) \abs{\nabla \mu}^{2} + n(\varphi) \abs{\chi_{\sigma}\nabla \sigma - \chi_{\varphi} \nabla \varphi}^{2} \dx + \int_{\Gamma} K \chi_{\sigma} \abs{\sigma}^{2} \dHaus \\
& \quad + \int_{\Omega}  h(\varphi) \left ( \lambda_{c} \sigma (\chi_{\sigma} \sigma + \chi_{\varphi}(1-\varphi)) - (\lambda_{p} \sigma - \lambda_{a}) \mu \right ) \dx \\
& \quad -  \int_{\Gamma} K (\chi_{\sigma} \sigma + \chi_{\varphi}(1-\varphi)) \sigma_{\infty} - K \chi_{\varphi}(1-\varphi)\sigma \dHaus = 0.
\end{aligned}
\end{equation}
We first estimate the mean $\mean{\mu}$ using \eqref{Truncated:mu:weak} by considering $\phi = 1$ and using the growth condition \eqref{assump:Psi}, leading to
\begin{align*}
\norm{\mean{\mu}}_{L^{2}}^{2} & = \abs{\mean{\mu}}^{2} \abs{\Omega} = \abs{\Omega}^{-1} \abs{\int_{\Omega} A \Psi'(\varphi) - \chi_{\varphi} \sigma \dx}^{2} \\
& \leq \abs{\Omega}^{-1} \left ( A R_{3} \abs{\Omega} + A R_{3} \norm{\varphi}_{L^{2}} \abs{\Omega}^{\frac{1}{2}} + \chi_{\varphi} \norm{\sigma}_{L^{2}} \abs{\Omega}^{\frac{1}{2}} \right )^{2} \\
& \leq 3 \abs{\Omega}^{-1} \left ( A^{2} R_{3}^{2} \abs{\Omega}^{2} + A^{2} R_{3}^{2} \norm{\varphi}_{L^{2}}^{2} \abs{\Omega} + \chi_{\varphi}^{2} \norm{\sigma}_{L^{2}}^{2} \abs{\Omega} \right ).
\end{align*}
Employing the Poincar\'{e} inequality \eqref{regular:Poincare} we have
\begin{equation}\label{L2norm:mu:Poincare}
\begin{aligned}
\norm{\mu}_{L^{2}}^{2} & \leq 2 C_{\mathrm{P}}^{2} \norm{\nabla \mu}_{L^{2}}^{2} + 2 \norm{\mean{\mu}}_{L^{2}}^{2} \\
& \leq  2 C_{\mathrm{P}}^{2} \norm{\nabla \mu}_{L^{2}}^{2} +  6  \left ( A^{2} R_{3}^{2} \abs{\Omega} + A^{2} R_{3}^{2} \norm{\varphi}_{L^{2}}^{2} + \chi_{\varphi}^{2} \norm{\sigma}_{L^{2}}^{2} \right ).
\end{aligned}
\end{equation}
Then, by H\"{o}lder's inequality and Young's inequality, we can estimate the source term involving $\mu$ as follows,
\begin{equation}\label{apriori:musourceterms}
\begin{aligned}
& \abs{ \int_{\Omega}  - h(\varphi) (\lambda_{p} \sigma - \lambda_{a}) \mu \dx}  \leq h_{\infty} \left ( \lambda_{p} \norm{\sigma}_{L^{2}} + \lambda_{a} \abs{\Omega}^{\frac{1}{2}} \right ) \norm{\mu}_{L^{2}} \\
& \quad  \leq  \frac{h_{\infty}^{2} \lambda_{p}^{2}}{4 a_{1}} \norm{\sigma}_{L^{2}}^{2} + C(a_{2}, \lambda_{a}, h_{\infty}, \abs{\Omega}) + (a_{1} + a_{2}) \norm{\mu}_{L^{2}}^{2} \\
& \quad \leq  2C_{\mathrm{P}}^{2} (a_{1} + a_{2}) \norm{\nabla \mu}_{L^{2}}^{2} + C( a_{1}, a_{2}, \lambda_{a}, h_{\infty}, \abs{\Omega}, A, R_{3}) \\
& \quad +  \left ( \frac{h_{\infty}^{2} \lambda_{p}^{2}}{4 a_{1}} + 6 (a_{1} + a_{2}) \chi_{\varphi}^{2} \right ) \norm{\sigma}_{L^{2}}^{2} + 6 A^{2} R_{3}^{2} (a_{1} + a_{2}) \norm{\varphi}_{L^{2}}^{2},
\end{aligned}
\end{equation}
for some positive constants $a_{1}$ and $a_{2}$ yet to be determined.  For the term involving $\lambda_{c}$, we obtain from H\"{o}lder's inequality and Young's inequality
\begin{equation}\label{apriori:lambdacterm}
\begin{aligned}
& \abs{\int_{\Omega} \lambda_{c} h(\varphi) \sigma (\chi_{\sigma} \sigma + \chi_{\varphi}(1-\varphi)) \dx} \\
& \leq \lambda_{c} h_{\infty} \left ( \chi_{\sigma} \norm{\sigma}_{L^{2}}^{2} + \chi_{\varphi} \norm{\varphi}_{L^{2}}\norm{\sigma}_{L^{2}} + \chi_{\varphi} \norm{\sigma}_{L^{1}} \right ) \\
& \leq \lambda_{c} h_{\infty} \left (  \chi_{\sigma} + a_{4} + \frac{a_{3}\chi_{\varphi}}{2} \right ) \norm{\sigma}_{L^{2}}^{2} + \lambda_{c} h_{\infty} \frac{\chi_{\varphi}}{2a_{3}} \norm{\varphi}_{L^{2}}^{2} + C(\abs{\Omega}, \lambda_{c}, h_{\infty}, \chi_{\sigma}, \chi_{\varphi}, a_{4}), 
\end{aligned}
\end{equation}
for some positive constants $a_{3}$ and $ a_{4}$ yet to be determined.  For the terms involving the boundary integral, we have by H\"{o}lder's inequality, Young's inequality and the trace theorem,
\begin{equation}\label{apriori:boundaryterm:tracethm}
\begin{aligned}
& \abs{\int_{\Gamma} \chi_{\varphi}(1-\varphi) \sigma - \chi_{\sigma} \sigma \sigma_{\infty} - \chi_{\varphi}(1-\varphi) \sigma_{\infty} \dHaus} \\
& \leq  \chi_{\varphi} \left ( \norm{\sigma}_{L^{1}(\Gamma)} + \norm{\varphi}_{L^{2}(\Gamma)} \norm{\sigma}_{L^{2}(\Gamma)} \right ) + \chi_{\sigma} \norm{\sigma}_{L^{2}(\Gamma)} \norm{\sigma_{\infty}}_{L^{2}(\Gamma)} \\
& + \chi_{\varphi} \norm{\sigma_{\infty}}_{L^{1}(\Gamma)} + \chi_{\varphi} \norm{\varphi}_{L^{2}(\Gamma)} \norm{\sigma_{\infty}}_{L^{2}(\Gamma)} \\
& \leq \left( a_{5} + \frac{\chi_{\sigma}}{2} \right ) \norm{\sigma}_{L^{2}(\Gamma)}^{2} + \left ( \frac{\chi_{\varphi}^{2}}{2 \chi_{\sigma}} + a_{6} \right ) \norm{\varphi}_{L^{2}(\Gamma)}^{2} + C(a_{5}, a_{6}, \chi_{\varphi}, \chi_{\sigma}, \abs{\Gamma}) \left (1 + \norm{\sigma_{\infty}}_{L^{2}(\Gamma)}^{2} \right ) \\
& \leq \left (  a_{5} + \frac{\chi_{\sigma}}{2}\right ) \norm{\sigma}_{L^{2}(\Gamma)}^{2}  + C_{\mathrm{tr}}^{2} \left (  \frac{\chi_{\varphi}^{2}}{2 \chi_{\sigma}} + a_{6} \right ) \norm{\varphi}_{H^{1}}^{2} + C \left (1 + \norm{\sigma_{\infty}}_{L^{2}(\Gamma)}^{2} \right ),
\end{aligned}
\end{equation}
for some positive constants $a_{5}$ and $a_{6}$ yet to be determined.  Here, $C_{\mathrm{tr}}$ is the constant from the trace theorem which depends only on $\Omega$ and $d$,
\begin{align*}
\norm{f}_{L^{2}(\Gamma)} \leq C_{\mathrm{tr}} \norm{f}_{H^{1}} \quad \forall f \in H^{1}.
\end{align*}
Employing the estimates \eqref{apriori:musourceterms}, \eqref{apriori:lambdacterm}, and \eqref{apriori:boundaryterm:tracethm}, and using the lower bounds of $m(\cdot)$ and $n(\cdot)$, we  obtain from \eqref{apriori:totalenergymean}
\begin{equation}\label{apriori:totalenergymean:2}
\begin{aligned}
& \frac{\dd}{\dt} \int_{\Omega} \left [ A \Psi(\varphi) + \frac{B}{2} \abs{\nabla \varphi}^{2} + \frac{\chi_{\sigma}}{2} \abs{\sigma}^{2} + \chi_{\varphi} \sigma (1-\varphi) \right ] \dx \\
& \quad + \int_{\Omega} \left ( m_{0} - 2C_{\mathrm{P}}^{2} (a_{1} + a_{2}) \right ) \abs{\nabla \mu}^{2} + n_{0} \abs{\chi_{\sigma} \nabla \sigma - \chi_{\varphi} \nabla \varphi}^{2} \dx \\
& \quad + K \int_{\Gamma} \left ( \chi_{\sigma} -  a_{5} - \frac{\chi_{\sigma}}{2} \right ) \abs{\sigma}^{2} \dHaus - K \int_{\Omega} C_{\mathrm{tr}}^{2} \left ( \frac{\chi_{\varphi}^{2}}{2 \chi_{\sigma}} + a_{6} \right ) \abs{\nabla \varphi}^{2} \dx \\
& \quad  -  \int_{\Omega} \left ( \frac{h_{\infty}^{2} \lambda_{p}^{2}}{4 a_{1}} + 6 (a_{1} + a_{2}) \chi_{\varphi}^{2} + \lambda_{c} h_{\infty} \left (  \chi_{\sigma} + a_{4} + \frac{a_{3}\chi_{\varphi}}{2} \right )\right ) \abs{\sigma}^{2} \dx \\
& \quad -  \int_{\Omega} \left ( 6 A^{2} R_{3}^{2} (a_{1} + a_{2}) + \lambda_{c} h_{\infty} \frac{\chi_{\varphi}}{2a_{3}} + K C_{\mathrm{tr}}^{2} \left (  \frac{\chi_{\varphi}^{2}}{2 \chi_{\sigma}} + a_{6} \right )   \right ) \abs{\varphi}^{2} \dx \\
& \quad \leq  C \left ( 1 + \norm{\sigma_{\infty}}_{L^{2}(\Gamma)}^{2} \right ),
\end{aligned}
\end{equation}
where $C$ is independent of $\varphi$, $\sigma$ and $\mu$.  By the triangle inequality, Minkowski's inequality and Young's inequality, we see that
\begin{align}\label{nablasigmaL2normFromNSigma}
\norm{\chi_{\sigma} \nabla \sigma}_{L^{2}}^{2} \leq \left ( \norm{\nabla N_{,\sigma}}_{L^{2}} + \norm{\chi_{\varphi} \nabla \varphi}_{L^{2}} \right )^{2} \leq 2 \norm{\nabla N_{,\sigma}}_{L^{2}}^{2} + 2 \norm{\chi_{\varphi} \nabla \varphi}_{L^{2}}^{2}.
\end{align}
We now choose the constants $\{a_{i}\}_{i = 1}^{6}$ to be
\begin{align*}
a_{1} = a_{2} = \frac{m_{0}}{8 C_{\mathrm{P}}^{2}}, \; a_{5} = \frac{\chi_{\sigma}}{4}, \; a_{3} = a_{4} = a_{6} = 1,
\end{align*}
and write
\begin{align*}
c_{1} &:= \frac{m_{0}}{2}, \quad c_{2} := K \frac{\chi_{\sigma}}{4}, \quad c_{3} :=  KC_{\mathrm{tr}}^{2} \left ( \frac{\chi_{\varphi}^{2}}{2 \chi_{\sigma}} + 1 \right ) + \chi_{\varphi}^{2} n_{0}, \\
c_{4} & := \frac{2 h_{\infty}^{2} \lambda_{p}^{2} C_{\mathrm{P}}^{2}}{m_{0}} + \frac{3 m_{0}}{2 C_{\mathrm{P}}^{2}}\chi_{\varphi}^{2} + \lambda_{c} h_{\infty} \left (  \chi_{\sigma} + 1 + \frac{\chi_{\varphi}}{2} \right ), \\
c_{5} & :=  \frac{3 m_{0}}{2 C_{\mathrm{P}}^{2}} A^{2} R_{3}^{2}  + \lambda_{c} h_{\infty} \frac{\chi_{\varphi}}{2} + K C_{\mathrm{tr}}^{2} \left (  \frac{\chi_{\varphi}^{2}}{2 \chi_{\sigma}} + 1 \right ),
\end{align*}
where the additional $\chi_{\varphi}^{2} n_{0}$ in the constant $c_{3}$ comes from \eqref{nablasigmaL2normFromNSigma}.  Then \eqref{apriori:totalenergymean:2} becomes
\begin{equation}\label{apriori:totaltimeenergymean:sim}
\begin{aligned}
& \frac{\dd}{\dt} \int_{\Omega} \left [ A \Psi(\varphi) + \frac{B}{2} \abs{\nabla \varphi}^{2} + \frac{\chi_{\sigma}}{2} \abs{\sigma}^{2} + \chi_{\varphi} \sigma (1-\varphi) \right ] \dx \\
& \quad + \int_{\Omega} c_{1} \abs{\nabla \mu}^{2} + \frac{n_{0} \chi_{\sigma}^{2}}{2} \abs{\nabla \sigma}^{2} \dx  + \int_{\Gamma} c_{2} \abs{\sigma}^{2} \dHaus \\
& \quad - \int_{\Omega} c_{4} \abs{\sigma}^{2} + c_{5} \abs{\varphi}^{2} + c_{3} \abs{\nabla \varphi}^{2} \dx  \leq C \left (1 + \norm{\sigma_{\infty}}_{L^{2}(\Gamma)}^{2} \right ).
\end{aligned}
\end{equation}
Integrating \eqref{apriori:totaltimeenergymean:sim} with respect to $t$ from $0$ to $s \in (0,T]$ gives
\begin{equation}\label{apriori:totaltimeenergymean:int}
\begin{aligned}
&  \int_{\Omega} \left [ A \Psi(\varphi(x,s)) + \frac{B}{2} \abs{\nabla \varphi(x, s)}^{2} + \frac{\chi_{\sigma}}{2} \abs{\sigma(x, s)}^{2} + \chi_{\varphi} \sigma(x, s)(1-\varphi(x,s)) \right ] \dx \\
& \quad + c_{1} \norm{\nabla \mu}_{L^{2}(0,s;L^{2})}^{2} + \frac{n_{0} \chi_{\sigma}^{2}}{2}  \norm{\nabla \sigma}_{L^{2}(0,s;L^{2})}^{2} +  c_{2} \norm{\sigma}_{L^{2}(0,s;L^{2}(\Gamma))}^{2} \\
& \quad - c_{4} \norm{\sigma}_{L^{2}(0,s;L^{2})}^{2} - c_{5} \norm{\varphi}_{L^{2}(0,s;L^{2})}^{2} - c_{3} \norm{\nabla \varphi}_{L^{2}(0,s;L^{2})}^{2} \\
& \quad \leq c_{0} + C \left (s + \norm{\sigma_{\infty}}_{L^{2}(0,s;L^{2}(\Gamma))}^{2} \right ),
\end{aligned}
\end{equation}
where the constant $c_{0}$ is defined in \eqref{defn:initialenergy}.  By H\"{o}lder's inequality and Young's inequality, we have
\begin{equation}\label{sigmavarphicrossterm}
\begin{aligned}
\abs{\int_{\Omega} \chi_{\varphi} \sigma (1-\varphi) \dx} & \leq \chi_{\varphi} \norm{\sigma}_{L^{1}} + \chi_{\varphi} \norm{\sigma}_{L^{2}} \norm{\varphi}_{L^{2}} \\
& \leq \frac{\chi_{\sigma}}{8} \norm{\sigma}_{L^{2}}^{2} + C(\chi_{\sigma}, \abs{\Omega}, \chi_{\varphi}) + \frac{\chi_{\sigma}}{8} \norm{\sigma}_{L^{2}}^{2} + \frac{2\chi_{\varphi}^{2}}{\chi_{\sigma}} \norm{\varphi}_{L^{2}}^{2},
\end{aligned}
\end{equation}
and thus from \eqref{apriori:totaltimeenergymean:int} we deduce that
\begin{equation}\label{apriori:totalenergymean:int:2}
\begin{aligned}
& A  \norm{\Psi(\varphi(s))}_{L^{1}} + \frac{B}{2} \norm{\nabla \varphi(s)}_{L^{2}}^{2} + \frac{\chi_{\sigma}}{4} \norm{\sigma(s)}_{L^{2}}^{2} - \frac{2\chi_{\varphi}^{2}}{\chi_{\sigma}} \norm{\varphi(s)}_{L^{2}}^{2} \\
& \quad + c_{1} \norm{\nabla \mu}_{L^{2}(0,s;L^{2})}^{2} + \frac{n_{0} \chi_{\sigma}^{2}}{2} \norm{\nabla \sigma}_{L^{2}(0,s;L^{2})}^{2} + c_{2} \norm{\sigma}_{L^{2}(0,s;L^{2}(\Gamma))}^{2} \\
& \quad - c_{4} \norm{\sigma}_{L^{2}(0,s;L^{2})}^{2} - c_{5} \norm{\varphi}_{L^{2}(0,s;L^{2})}^{2} - c_{3} \norm{\nabla \varphi}_{L^{2}(0,s;L^{2})}^{2} \\
& \quad \leq c_{0} + C \left (1 + T + \norm{\sigma_{\infty}}_{L^{2}(0,T;L^{2}(\Gamma))}^{2} \right ).
\end{aligned}
\end{equation}
Now, by \eqref{assump:Psi}, we have
\begin{align}\label{varphiL2PsiL1}
\norm{\varphi}_{L^{2}}^{2} = \int_{\Omega} \abs{\varphi}^{2} \dx \leq \frac{1}{R_{1}} \left ( \int_{\Omega} \Psi(\varphi) \dx + R_{2} \abs{\Omega} \right ) = \frac{1}{R_{1}} \norm{\Psi(\varphi)}_{L^{1}} + \frac{R_{2}}{R_{1}} \abs{\Omega},
\end{align}
and, for any $s \in (0,T]$,
\begin{align}\label{varphiL2L2PsiL1L1}
\norm{\varphi}_{L^{2}(0,s;L^{2})}^{2} \leq \frac{1}{R_{1}} \norm{\Psi(\varphi)}_{L^{1}(0,s;L^{1})} + \frac{R_{2}}{R_{1}} \abs{\Omega} s.
\end{align}
Thus, using \eqref{varphiL2PsiL1} and \eqref{varphiL2L2PsiL1L1}, we obtain from \eqref{apriori:totalenergymean:int:2}
\begin{equation}\label{apriori:totalenergymean:int:3}
\begin{aligned}
& \left (A - \frac{2 \chi_{\varphi}^{2}}{\chi_{\sigma} R_{1}} \right )  \norm{\Psi(\varphi(s))}_{L^{1}} +  \frac{B}{2} \norm{\nabla \varphi(s)}_{L^{2}}^{2} + \frac{\chi_{\sigma}}{4} \norm{\sigma(s)}_{L^{2}}^{2} \\
& \quad - \frac{c_{5}}{R_{1}}  \norm{\Psi(\varphi(s))}_{L^{1}(0,s;L^{1})} -  c_{3} \norm{\nabla \varphi}_{L^{2}(0,s;L^{2})}^{2} - c_{4} \norm{\sigma}_{L^{2}(0,s;L^{2})}^{2} \\
& \quad + c_{1} \norm{\nabla \mu}_{L^{2}(0,s:L^{2})}^{2} + \frac{n_{0} \chi_{\sigma}^{2}}{2} \norm{\nabla \sigma}_{L^{2}(0,s;L^{2})}^{2} +  c_{2} \norm{\sigma}_{L^{2}(0,s;L^{2}(\Gamma))}^{2} \\
& \quad \leq C \left (1 + T + \norm{\sigma_{\infty}}_{L^{2}(0,T;L^{2}(\Gamma))}^{2} \right ) =: c_{*},
\end{aligned}
\end{equation}
for some positive constant $c_{*}$ independent of $s \in (0,T]$, $\mu(s)$, $\sigma(s)$, and $\varphi(s)$.  Let 
\begin{align*}
c_{\min}  := \min \left ( A - \frac{2 \chi_{\varphi}^{2}}{\chi_{\sigma}R_{1}} , \frac{B}{2},\frac{\chi_{\sigma}}{4} \right ), \quad c_{\max}  := \max (c_{5}/ R_{1}, c_{3}, c_{4} ).
\end{align*}
Then, $c_{\min} > 0$ by assumption (see \eqref{assump:constantsrelations}), and we obtain from \eqref{apriori:totalenergymean:int:3} that
\begin{equation}\label{apriori:totalenergymean:int:4}
\begin{aligned}
& c_{\min}  \left ( \norm{\Psi(\varphi(s))}_{L^{1}} + \norm{\nabla \varphi(s)}_{L^{2}}^{2} + \norm{\sigma(s)}_{L^{2}}^{2} \right ) \\
& \quad + c_{1} \norm{\nabla \mu}_{L^{2}(0,s;L^{2})}^{2} + \frac{n_{0} \chi_{\sigma}^{2}}{2} \norm{\nabla \sigma }_{L^{2}(0,s;L^{2})}^{2} +  c_{2} \norm{\sigma}_{L^{2}(0,s;L^{2}(\Gamma))}^{2} \\
& \quad \leq \int_{0}^{s} c_{\max} \left ( \norm{\Psi(\varphi)}_{L^{1}} + \norm{\nabla \varphi}_{L^{2}}^{2} + \norm{\sigma}_{L^{2}}^{2} \right ) \dt + c_{*}.
\end{aligned}
\end{equation}
Substituting
\begin{align*}
u(s) & = \norm{\Psi(\varphi(s))}_{L^{1}} + \norm{\nabla \varphi(s)}_{L^{2}}^{2} + \norm{\sigma(s)}_{L^{2}}^{2}, \\
v(t) & = \frac{1}{c_{\min}} \left ( c_{1} \norm{\nabla \mu}_{L^{2}}^{2} + \frac{n_{0} \chi_{\sigma}^{2}}{2} \norm{\nabla \sigma}_{L^{2}}^{2} +  c_{2} \norm{\sigma}_{L^{2}(\Gamma)}^{2} \right ), \\
\alpha(s) & = \frac{c_{*}}{c_{\min}}, \quad \beta(t) = \frac{c_{\max}}{c_{\min}}
\end{align*}
into Lemma \ref{lem:GronwallInt}, we obtain from \eqref{apriori:totalenergymean:int:4}
\begin{equation}\label{apriori:Linftybddintime}
\begin{aligned}
& \norm{\Psi(\varphi(s))}_{L^{1}} + \norm{\nabla \varphi(s)}_{L^{2}}^{2} + \norm{\sigma(s)}_{L^{2}}^{2}  \\
& \quad + \frac{1}{c_{\min}} \left ( c_{1} \norm{\nabla \mu}_{L^{2}(0,s;L^{2})}^{2} + \frac{n_{0} \chi_{\sigma}^{2}}{2} \norm{\nabla \sigma}_{L^{2}(0,s;L^{2})}^{2} +  c_{2} \norm{\sigma}_{L^{2}(0,s;L^{2}(\Gamma))}^{2} \right ) \\
& \quad \leq \frac{c_{*}}{c_{\min}} + \int_{0}^{s}  \frac{c_{*} c_{\max}}{c_{\min}^{2}} \exp \left ( \frac{c_{\max}}{c_{\min}}(s-t) \right ) \dt < \infty \quad \forall s \in (0,T].
\end{aligned}
\end{equation}
Together with \eqref{varphiL2PsiL1}, we find that there exists a positive constant $\overline{C}$ not depending on $\varphi$, $\mu$ and $\sigma$ such that
\begin{equation}
\begin{aligned}
& \norm{\Psi(\varphi(s))}_{L^{1}} + \norm{\varphi(s)}_{H^{1}}^{2} + \norm{\sigma(s)}_{L^{2}}^{2} \\
& \quad + \norm{\nabla \mu}_{L^{2}(0,s;L^{2})}^{2} + \norm{\nabla \sigma}_{L^{2}(0,s;L^{2})}^{2} + \norm{\sigma}_{L^{2}(0,s;L^{2}(\Gamma))}^{2} \leq \overline{C},
\end{aligned}
\end{equation}
for all $s \in (0,T]$.
\end{proof}

\begin{remark}\label{remark:constantrelations}
The necessity of \eqref{assump:constantsrelations} comes from the fact that in \eqref{apriori:totaltimeenergymean:sim}, we cannot apply H\"{o}lder's inequality and Young's inequality like in \eqref{sigmavarphicrossterm} to estimate the term 
\begin{align*}
\frac{\dd}{\dt} \int_{\Omega} \chi_{\varphi} \sigma(1-\varphi) \dx,
\end{align*}
as inequalities are not preserved under differentiation.  
\end{remark}

\section{Global weak solutions}\label{sec:Existence}
\subsection{Galerkin approximation}\label{sec:Galerkinapprox}
We obtain global weak solutions via a suitable Galerkin procedure.  Consider a basis $\{ w_{i} \}_{i \in \N}$ of $H^{1}$ which is orthonormal with respect to the $L^{2}$-inner product, and, without loss of generality, we assume $w_{1}$ is constant and hence $\int_{\Omega} w_{i} \dx = 0$ for all $i \geq 2$.  In the following we take $\{w_{i}\}_{i \in \N}$ to be eigenfunctions for the Laplacian with homogeneous Neumann boundary conditions,
\begin{subequations}\label{NeuLaplace}
\begin{alignat}{2}
-\Laplace w_{i} & = \Lambda_{i} w_{i} && \text{ in } \Omega, \\
\nabla w_{i} \cdot \nu & = 0 && \text{ on } \Gamma,
\end{alignat}
\end{subequations}
where $\Lambda_{i}$ is the eigenvalue corresponding to $w_{i}$.  It is well-known that the $\{w_{i}\}_{i \in \N}$ can be chosen as an orthonormal basis of $L^{2}$ and then forms an orthogonal basis of $H^{1}$.  As constant functions are eigenfunctions, $w_{1}$ can be chosen as a constant function with $\Lambda_{1} = 0$ (see for instance \cite[Theorem 8.4]{book:Salsa}).  Let
\begin{align*}
W_{k} := \mathrm{span} \{ w_{1}, \dots, w_{k} \} \subset H^{1}
\end{align*}
denote the finite dimensional space spanned by the first $k$ basis functions.  We now consider
\begin{subequations}\label{Galerkin:ansatz}
\begin{align}
\varphi_{k}(t,x) =\sum_{i=1}^{k} \alpha_{i}^{k}(t) w_{i}(x), \; \mu_{k}(t,x) = \sum_{i=1}^{k} \beta_{i}^{k}(t) w_{i}(x), \; \sigma_{k}(t,x) = \sum_{i=1}^{k} \gamma_{i}^{k}(t) w_{i}(x),
\end{align}
\end{subequations}
and the following Galerkin ansatz,
\begin{subequations}\label{Galerkin:system}
\begin{align}
\int_{\Omega} \pd_{t} \varphi_{k} w_{j} \dx & = \int_{\Omega} - m(\varphi_{k}) \nabla \mu_{k} \cdot \nabla w_{j} + (\lambda_{p} \sigma_{k} - \lambda_{a}) h(\varphi_{k}) w_{j} \dx, \label{Galerkin:varphi} \\
\int_{\Omega} \mu_{k} w_{j} \dx & = \int_{\Omega} A \Psi'(\varphi_{k}) w_{j} + B \nabla \varphi_{k} \cdot \nabla w_{j} - \chi_{\varphi} \sigma_{k} w_{j} \dx,  \label{Galerkin:mu} \\
\int_{\Omega} \pd_{t} \sigma_{k} w_{j} \dx & = \int_{\Omega} -n(\varphi_{k}) (\chi_{\sigma} \nabla \sigma_{k} - \chi_{\varphi} \nabla \varphi_{k}) \cdot \nabla w_{j} - \lambda_{c} \sigma_{k} h(\varphi_{k}) w_{j} \dx  \label{Galerkin:sigma} \\
\notag & + \int_{\Gamma} K ( \sigma_{\infty} - \sigma_{k}) w_{j} \dHaus,
\end{align}
\end{subequations}
for $1 \leq j \leq k$.  We define the following symmetric matrices with components
\begin{equation*}
\begin{alignedat}{3}
(\bm{M}_{h}^{k})_{ji} & := \int_{\Omega} h(\varphi_{k}) w_{i} w_{j} \dx, &&\quad (\bm{M}_{\Gamma})_{ji} && := \int_{\Gamma} w_{i} w_{j} \dHaus, \\
(\bm{S}_{m}^{k})_{ji} & := \int_{\Omega} m(\varphi_{k}) \nabla w_{i} \cdot \nabla w_{j} \dx, && \quad (\bm{S}_{n}^{k})_{ji} &&:= \int_{\Omega} n(\varphi_{k}) \nabla w_{i} \cdot \nabla w_{j} \dx,
\end{alignedat}
\end{equation*}
for $1 \leq i, j \leq k$.  Let $\delta_{ij}$ denote the Kronecker delta, and we introduce the notation
\begin{equation*}
\begin{alignedat}{5}
\psi_{j}^{k} & := \int_{\Omega} \Psi'(\varphi_{k})w_{j} \dx , && \quad \Sigma_{j}^{k} && := \int_{\Gamma} \sigma_{\infty} w_{j} \dHaus, && \quad h_{j}^{k} := \int_{\Omega} h(\varphi_{k}) w_{j} \dx, \\
\bm{\psi}^{k} & := (\psi_{1}^{k}, \dots, \psi_{k}^{k})^{\top}, && \quad \bm{\Sigma}^{k} && := (\Sigma_{1}^{k}, \dots, \Sigma_{k}^{k})^{\top}, && \quad \bm{h}^{k} := (h_{1}^{k}, \dots, h_{k}^{k})^{\top}, \\
\bm{M}_{ij} & = \int_{\Omega} w_{i} w_{j} \dx = \delta_{ij}, && \quad \bm{S}_{ij} && := \int_{\Omega} \nabla w_{i} \cdot \nabla w_{j} \dx, && \quad &&
\end{alignedat}
\end{equation*}
for $1 \leq i,j \leq k$, so that we obtain the following initial value problem for a system of ordinary differential equations for $\bm{\alpha}_{k} := (\alpha_{1}^{k}, \dots \alpha_{k}^{k})^{\top}$, $\bm{\beta}_{k} := (\beta_{1}^{k}, \dots, \beta_{k}^{k})^{\top}$, and $\bm{\gamma}_{k} := (\gamma_{1}^{k}, \dots, \gamma_{k}^{k})^{\top}$,
\begin{subequations}\label{discrete:system1}
\begin{align}
\frac{\dd}{\dt} \bm{\alpha}_{k} & = - \bm{S}_{m}^{k} \bm{\beta}_{k} + \lambda_{p} \bm{M}_{h}^{k} \bm{\gamma}_{k} - \lambda_{a} \bm{h}^{k}, \label{discrete:varphi} \\
\bm{\beta}_{k} & = A \bm{\psi}^{k} + B\bm{S} \bm{\alpha} - \chi_{\varphi} \bm{\gamma}_{k}, \label{discrete:mu} \\
\frac{\dd}{\dt} \bm{\gamma}_{k} & = - \bm{S}_{n}^{k}(\chi_{\sigma} \bm{\gamma}_{k} - \chi_{\varphi} \bm{\alpha}_{k}) - \lambda_{c} \bm{M}_{h}^{k} \bm{\gamma}_{k} - K\bm{M}_{\Gamma}\bm{\gamma}_{k} + K \bm{\Sigma}^{k}.  \label{discrete:sigma}
\end{align}
\end{subequations}
Substituting \eqref{discrete:mu} into \eqref{discrete:varphi}, we obtain
\begin{subequations}\label{discrete:system:combined}
\begin{align}
\frac{\dd}{\dt} \bm{\alpha}_{k} & = - \bm{S}_{m}^{k}  (A \bm{\psi}^{k} + B \bm{S} \bm{\alpha}_{k} - \chi_{\varphi} \bm{\gamma}_{k}) +  \lambda_{p} \bm{M}_{h}^{k} \bm{\gamma}_{k} - \lambda_{a} \bm{h}^{k}, \\
\frac{\dd}{\dt} \bm{\gamma}_{k} & = -  \bm{S}_{n}^{k}(\chi_{\sigma} \bm{\gamma}_{k} - \chi_{\varphi} \bm{\alpha}_{k}) - \lambda_{c} \bm{M}_{h}^{k} \bm{\gamma}_{k} - K \bm{M}_{\Gamma}\bm{\gamma}_{k} + K \bm{\Sigma}^{k}, 
\end{align}
\end{subequations}
and we complete \eqref{discrete:system:combined} with the initial conditions
\begin{align}\label{discrete:system:combined:initial}
(\bm{\alpha}_{k})_{j}(0) = \int_{\Omega} \varphi_{0} w_{j} \dx, \quad (\bm{\gamma}_{k})_{j}(0) = \int_{\Omega} \sigma_{0} w_{j} \dx \text{ for } 1 \leq j \leq k,
\end{align}
which satisfy
\begin{align*}
\bignorm{\sum_{i=1}^{k} (\bm{\alpha}_{k})_{i}(0) w_{i}}_{H^{1}} \leq \norm{\varphi_{0}}_{H^{1}}, \quad \bignorm{\sum_{j=1}^{k} (\bm{\gamma}_{k})_{i}(0) w_{i}}_{L^{2}} \leq \norm{\sigma_{0}}_{L^{2}} \quad \forall k \in \N.
\end{align*}
We remark that \eqref{discrete:system:combined} is a nonlinear ODE system and $\bm{S}_{m}^{k}$, $\bm{S}_{n}^{k}$, $\bm{\psi}^{k}$, $\bm{M}_{h}^{k}$ depend in a nonlinear way on the solution.  Continuity of $m(\cdot)$, $n(\cdot)$, $h(\cdot)$ and $\Psi'(\cdot)$ imply that the right-hand sides of \eqref{discrete:system:combined} depend continuously on $\bm{\alpha}_{k}$ and $\bm{\gamma}_{k}$.  Thus, we can appeal to the theory of ODEs (via the Cauchy--Peano theorem) to infer that the initial value problem \eqref{discrete:system:combined} has at least one local solution pair $(\bm{\alpha}_{k}, \bm{\gamma}_{k})$ defined on $[0,t_{k}]$ for each $k \in \N$.  

\subsection{A priori estimates}
Next, we show that $t_{k} = T$ for each $k \in \N$ by deriving a priori estimates.  By the Cauchy--Peano theorem, \eqref{discrete:mu}, and \eqref{Galerkin:ansatz}, we see that
\begin{align*}
\varphi_{k}, \sigma_{k} \in C^{1}([0,t_{k}];W_{k}), \quad \mu_{k} \in C^{0}([0,t_{k}]; W_{k}).
\end{align*}
We proceed similarly to the derivation of \eqref{apriori:main}.  Let $\delta_{ij}$ denote the Kronecker delta.  Multiplying \eqref{Galerkin:sigma} with $\chi_{\sigma} \gamma_{j}^{k} + \chi_{\varphi}( w_{1}^{-1} \delta_{1j}-\alpha_{j}^{k})$ and summing from $j=1$ to $k$ leads to
\begin{equation}
\label{discrete:apriori:sigma}
\begin{aligned}
& \int_{\Omega} \pd_{t} \sigma_{k} (\chi_{\sigma} \sigma_{k} + \chi_{\varphi} (1 - \varphi_{k})) + n(\varphi_{k}) \abs{\chi_{\sigma} \nabla \sigma_{k} - \chi_{\varphi} \nabla \varphi_{k}}^{2} \dx \\
& \quad = -\int_{\Omega} \lambda_{c} \sigma_{k} h(\varphi_{k})(\chi_{\sigma} \sigma_{k} + \chi_{\varphi}(1-\varphi_{k})) \dx \\
& \quad + \int_{\Gamma} K (\sigma_{\infty} - \sigma_{k})(\chi_{\sigma} \sigma_{k} + \chi_{\varphi}(1-\varphi_{k})) \dHaus.
\end{aligned}
\end{equation}
Here, we used that $w_{1}$ is constant, $\nabla w_{1} = \bm{0}$, and the linearity of the trace operator.  Next, we multiply \eqref{Galerkin:varphi} with $\beta_{j}^{k}$, and summing the product from $j = 1$ to $k$ leads to
\begin{align}\label{discrete:apriori:varphi}
& \; \int_{\Omega} (\pd_{t}\varphi_{k} - \lambda_{p} \sigma_{k} h(\varphi_{k}) + \lambda_{a} h(\varphi_{k})) \mu_{k}  + m(\varphi_{k}) \abs{\nabla \mu_{k}}^{2} \dx = 0.
\end{align}
Similarly, we multiply \eqref{Galerkin:mu} with $\frac{\dd}{\dt} \alpha_{j}^{k}$, and summing the product from $j = 1$ to $k$ gives
\begin{align}\label{discrete:apriori:mu}
0 = \int_{\Omega} (- \mu_{k} + A \Psi'(\varphi_{k}) - \chi_{\varphi} \sigma_{k} )\pd_{t} \varphi_{k} + B \nabla \varphi_{k} \cdot \nabla \pd_{t}\varphi_{k}  \dx.
\end{align}
Upon adding \eqref{discrete:apriori:sigma}, \eqref{discrete:apriori:varphi}, and \eqref{discrete:apriori:mu} we obtain
\begin{equation}
\label{discrete:apriori:1}
\begin{aligned}
& \frac{\dd}{\dt} \int_{\Omega} \left [ A \Psi(\varphi_{k}) + \frac{B}{2} \abs{\nabla \varphi_{k}}^{2} + \frac{\chi_{\sigma}}{2} \abs{\sigma_{k}}^{2} + \chi_{\varphi} \sigma_{k} (1-\varphi_{k}) \right ] \dx \\
& \quad + \int_{\Omega} m(\varphi_{k}) \abs{\nabla \mu_{k}}^{2} + n(\varphi_{k}) \abs{\chi_{\sigma} \nabla \sigma_{k} - \chi_{\varphi} \nabla \varphi_{k}}^{2} \dx + \int_{\Gamma} K \chi_{\sigma} \abs{\sigma_{k}}^{2}  \dHaus \\
& \quad + \int_{\Omega} \lambda_{c} \sigma_{k} h(\varphi_{k})(\chi_{\sigma} \sigma_{k} + \chi_{\varphi} (1-\varphi_{k})) + (\lambda_{a} - \lambda_{p} \sigma_{k} )h(\varphi_{k})\mu_{k} \dx \\
& \quad - \int_{\Gamma} K \sigma_{\infty} (\chi_{\sigma} \sigma_{k} + \chi_{\varphi} (1-\varphi_{k})) - K \sigma_{k} \chi_{\varphi} (1-\varphi_{k}) \dHaus = 0.
\end{aligned}
\end{equation}
Thanks to Young's inequality, Poincar\'{e} inequality and the trace theorem, we can deduce that an analogue of \eqref{apriori:totaltimeenergymean:sim} holds for $\varphi_{k}$, $\sigma_{k}$ and $\mu_{k}$ via a similar calculation to that in the proof of Lemma \ref{lem:MainAprioriEst:1}.  Then, following the proof of Lemma \ref{lem:MainAprioriEst:1}, we obtain the following discrete a priori estimate
\begin{equation}\label{discrete:apriori:est1}
\begin{aligned}
\sup_{s \in [0,T]} & \left ( \norm{\Psi(\varphi_{k}(s))}_{L^{1}} + \norm{\varphi_{k}(s)}_{H^{1}}^{2} + \norm{\sigma_{k}(s)}_{L^{2}}^{2} \right ) \\
& + \norm{\nabla \mu_{k}}_{L^{2}(0,T;L^{2})}^{2} +  \norm{\nabla \sigma_{k}}_{L^{2}(0,T;L^{2})}^{2} + \norm{\sigma_{k}}_{L^{2}(0,T;L^{2}(\Gamma))}^{2} \leq \overline{C},
\end{aligned}
\end{equation}
where $\overline{C}$ is the constant in Lemma \ref{lem:MainAprioriEst:1}.  Setting $j = 1$ in \eqref{Galerkin:mu} leads to
\begin{align*}
\int_{\Omega} \mu_{k} \dx = \int_{\Omega} A \Psi'(\varphi_{k}) - \chi_{\varphi} \sigma_{k} \dx,
\end{align*}
and applying the same calculation to that in \eqref{L2norm:mu:Poincare} we obtain analogously
\begin{equation}\label{muL2estimate}
\begin{aligned}
 \norm{\mu_{k}}_{L^{2}}^{2} & \leq 2 C_{\mathrm{P}}^{2} \norm{\nabla \mu_{k}}_{L^{2}}^{2} + 2 \abs{\int_{\Omega} \mu_{k} \dx}^{2} \abs{\Omega}^{-1} \\
& \leq 2 C_{\mathrm{P}}^{2} \norm{\nabla \mu_{k} }_{L^{2}}^{2} + 6 A^{2} R_{3}^{2} \norm{\varphi_{k}}_{L^{2}}^{2} + 6 \chi_{\varphi}^{2}  \norm{\sigma_{k}}_{L^{2}}^{2} + C(A, R_{3}, \abs{\Omega}).
\end{aligned}
\end{equation}
Integrating with respect to time from $0$ to $T$, and using \eqref{discrete:apriori:est1}, we obtain
\begin{equation}\label{apriori:mu}
\begin{aligned}
\norm{\mu_{k}}_{L^{2}(0,T;L^{2})}^{2} & \leq C\left ( \norm{\nabla \mu_{k}}_{L^{2}(0,T;L^{2})}^{2} + \norm{\varphi_{k}}_{L^{2}(0,T;L^{2})}^{2} + \norm{\sigma_{k}}_{L^{2}(0,T;L^{2})}^{2} + 1 \right ) \\
& \leq C(1 + \overline{C}).
\end{aligned}
\end{equation}
Thus, with \eqref{discrete:apriori:est1} and \eqref{apriori:mu}, we see that there exists a positive constant $C$ depending on $\overline{C}$ and $T$ such that
\begin{align*}
\sup_{s \in (0,T]} \norm{\varphi_{k}(s)}_{H^{1}} + \norm{\mu_{k}}_{L^{2}(0,T;H^{1})} + \norm{\sigma_{k}}_{L^{2}(0,T;H^{1})} \leq C
\end{align*}
for all $k$.  This a priori estimate in turn guarantees that the solution $\{\varphi_{k}, \sigma_{k}, \mu_{k} \}$ to \eqref{discrete:system:combined} can be extended to the interval $[0,T]$, and thus $t_{k} = T$ for each $k \in \N$.

\subsection{Passing to the limit}
Let $\Pi_{k}$ denote the orthogonal projection onto $W_{k} = \mathrm{span}\{w_{1}, \dots, w_{k}\}$.  Then,  for any $\zeta \in L^{2}(0,T;H^{1})$, we see that
\begin{align*}
\int_{\Omega} \pd_{t}\varphi_{k} \zeta \dx = \int_{\Omega} \pd_{t}\varphi_{k} \Pi_{k} \zeta \dx = \sum_{j=1}^{k} \int_{\Omega} \pd_{t}\varphi_{k} \zeta_{kj} w_{j} \dx,
\end{align*}
where $\{\zeta_{kj}\}_{1 \leq j \leq k} \subset \R^{k}$ are the coefficients such that $\Pi_{k} \zeta = \sum_{j=1}^{k} \zeta_{kj} w_{j}$.  Thus, from \eqref{Galerkin:varphi}, and the boundedness of $m(\cdot)$ and $h(\cdot)$, we find that
\begin{equation}\label{pdtvarphik:apriori}
\begin{aligned}
\abs{\int_{0}^{T} \int_{\Omega} \pd_{t}\varphi_{k} \zeta \dx} & \leq m_{1} \norm{\nabla \mu_{k}}_{L^{2}(\Omega \times (0,T))} \norm{\nabla \Pi_{k} \zeta}_{L^{2}(\Omega \times (0,T))} \\
& + h_{\infty} \left ( \lambda_{p} \norm{\sigma_{k}}_{L^{2}(\Omega \times (0,T))} + \lambda_{a} \abs{\Omega}^{\frac{1}{2}} T^{\frac{1}{2}} \right )\norm{\Pi_{k} \zeta}_{L^{2}(\Omega \times (0,T))} \\
& \leq C \norm{\zeta}_{L^{2}(0,T;H^{1})},
\end{aligned}
\end{equation}
for some constant $C > 0$ independent of $k$.  Similarly, we obtain from \eqref{Galerkin:sigma} that
\begin{align*}
\abs{ \int_{0}^{T} \int_{\Omega} \pd_{t}\sigma_{k} \zeta \dx} & \leq n_{1} \left ( \chi_{\sigma} \norm{\nabla \sigma_{k}}_{L^{2}(\Omega \times (0,T))} + \chi_{\varphi} \norm{\nabla \varphi_{k}}_{L^{2}(\Omega \times (0,T))} \right ) \norm{\nabla \Pi_{k} \zeta}_{L^{2}(\Omega \times (0,T))} \\
& + \lambda_{c} h_{\infty} \norm{\sigma_{k}}_{L^{2}(\Omega \times (0,T))} \norm{\Pi_{k} \zeta}_{L^{2}(\Omega \times (0,T))} \\
& + K C_{\mathrm{tr}} \left ( \norm{\sigma_{\infty}}_{L^{2}(\Gamma \times (0,T))} + \norm{\sigma_{k}}_{L^{2}(\Gamma \times (0,T))} \right ) \norm{\Pi_{k} \zeta}_{L^{2}(0,T;H^{1})} \\
& \leq C \norm{\zeta}_{L^{2}(0,T;H^{1})},
\end{align*}
for some constant $C > 0$ independent of $k$.  Hence, together with \eqref{discrete:apriori:est1} and \eqref{apriori:mu}, we find that
\begin{align*}
\{\varphi_{k}\}_{k \in \N} & \text{ bounded in } L^{\infty}(0,T;H^{1}) \cap H^{1}(0,T;(H^{1})^{*}), \\
\{\mu_{k}\}_{k \in \N } & \text{ bounded in } L^{2}(0,T;H^{1}), \\
\{\sigma_{k}\}_{k \in \N} &\text{ bounded in } L^{\infty}(0,T;L^{2}) \cap L^{2}(0,T;H^{1}) \cap L^{2}(0,T;L^{2}(\Gamma)) \cap H^{1}(0,T;(H^{1})^{*}).
\end{align*}
By standard compactness results (Banach--Alaoglu theorem and reflexive weak compactness theorem) and \cite[\S 8, Corollary 4]{article:Simon86}, we obtain, for a relabelled subsequence,
\begin{alignat*}{3}
\varphi_{k} & \rightarrow \varphi && \quad \text{ weakly-}* && \quad \text{ in } L^{\infty}(0,T;H^{1}), \\
\varphi_{k} & \rightarrow \varphi && \quad \text{ strongly } && \quad \text{ in } C([0,T];L^{p}) \cap L^{2}(0,T;L^{p}) \text{ and a.e. in } \Omega \times (0,T), \\
\pd_{t}\varphi_{k} & \rightarrow \pd_{t}\varphi && \quad \text{ weakly } && \quad \text{ in } L^{2}(0,T;(H^{1})^{*}), \\
\sigma_{k} & \rightarrow \sigma && \quad \text{ weakly-}* && \quad \text{ in } L^{2}(0,T;H^{1}) \cap L^{\infty}(0,T;L^{2}) \cap L^{2}(0,T;L^{2}(\Gamma)), \\
\sigma_{k} & \rightarrow \sigma && \quad \text{ strongly } && \quad \text{ in } L^{2}(0,T;L^{p}) \text{ and a.e. in } \Omega \times (0,T), \\
\pd_{t}\sigma_{k} & \rightarrow \pd_{t}\sigma && \quad \text{ weakly } && \quad \text{ in } L^{2}(0,T;(H^{1})^{*}), \\
\mu_{k} & \rightarrow \mu && \quad \text{ weakly } && \quad \text{ in } L^{2}(0,T;H^{1}), 
\end{alignat*}
where $p \in [1,\infty)$ for dimensions $d = 1,2$ and $p \in \left [1,\frac{2d}{d-2} \right )$ for dimensions $d \geq 3$.  In particular, the above compactness holds for $p \in [1,2]$ in any dimension $d$, i.e., $\varphi_{k} \to \varphi$ strongly in $L^{2}(0,T;L^{2}) \cong L^{2}(\Omega \times (0,T))$.

For a fixed $j$ and $\delta \in C^{\infty}_{c}(0,T)$, we have $\delta(t) w_{j} \in L^{2}(0,T;H^{1})$, and so, by the triangle inequality and H\"{o}lder's inequality, we obtain
\begin{align*}
\int_{0}^{T} \int_{\Omega} \abs{( \abs{\varphi_{k}} - \abs{\varphi}) (\delta w_{j})} \dx \dt \leq \norm{\varphi_{k} - \varphi}_{L^{2}(0,T;L^{2})} \norm{\delta w_{j}}_{L^{2}(0,T;L^{2})} \to 0 \text{ as } k \to \infty.
\end{align*}
In particular, we have
\begin{align*}
(1 + \abs{\varphi_{k}}) \abs{\delta w_{j}} \to (1+\abs{\varphi}) \abs{\delta w_{j}} \text{ strongly in } L^{1}(\Omega \times (0,T)) \text{ as } k \to \infty.
\end{align*}
By continuity and the growth assumptions on $\Psi'(\cdot)$, we have 
\begin{align*}
\Psi'(\varphi_{k}) \to \Psi'(\varphi) \text{ a.e. as } k \to \infty, \quad \abs{\Psi'(\varphi_{k}) \delta w_{j}} \leq R_{3} (1+\abs{\varphi_{k}}) \abs{\delta w_{j}}.
\end{align*} 
Then, the generalised Lebesgue dominated convergence theorem (see \cite[Theorem 1.9, p. 89]{book:Royden}, or \cite[Theorem 1.23, p. 59]{book:Alt}) yields that
\begin{align*}
\Psi'(\varphi_{k}) \delta w_{j} \to \Psi'(\varphi) \delta w_{j} \text{ strongly in } L^{1}(\Omega \times (0,T)) \text{ as } k \to \infty,
\end{align*}
which leads to
\begin{align*}
\int_{0}^{T} \int_{\Omega} \Psi'(\varphi_{k}) \delta w_{j} \dx \dt \to \int_{0}^{T} \int_{\Omega} \Psi'(\varphi) \delta w_{j} \dx \dt \text{ as } k \to \infty.
\end{align*}
Next, by continuity and boundedness of $m(\cdot)$, we see that $m(\varphi_{k}) \to m(\varphi)$ a.e. in $\Omega \times (0,T)$, and applying Lebesgue dominated convergence theorem to $(m(\varphi_{k}) - m(\varphi)) \abs{\delta \nabla w_{j}}$ yields
\begin{align*}
\norm{(m(\varphi_{k}) - m(\varphi)) \delta \nabla w_{j} }_{L^{2}(\Omega \times (0,T))} \to 0 \text{ as } k \to \infty.
\end{align*}
Together with the weak convergence $\nabla \mu_{k} \rightharpoonup \nabla \mu$ in $L^{2}(0,T;L^{2})$, we obtain, by the product of weak-strong convergence,
\begin{align*}
\int_{0}^{T} \int_{\Omega} m(\varphi_{k}) \delta \nabla w_{j} \cdot \nabla \mu_{k} \dx \dt \to \int_{0}^{T} \int_{\Omega} m(\varphi) \delta \nabla w_{j} \cdot \nabla \mu \dx \dt \text{ as } k \to \infty.
\end{align*}
Terms involving $n(\cdot)$ and $h(\cdot)$ can be dealt with in a similar fashion. 

Multiplying \eqref{Galerkin:system} with $\delta \in C^{\infty}_{c}(0,T)$, integrating in time from $0$ to $T$, and passing to the limit $k \to \infty$, we obtain
\begin{equation*}
\begin{aligned}
\int_{0}^{T} \delta(t) \inner{\pd_{t}\varphi}{w_{j}} \dt & = \int_{0}^{T} \int_{\Omega} \delta(t) \left ( -m(\varphi) \nabla \mu \cdot \nabla w_{j} + (\lambda_{p} \sigma - \lambda_{a}) h(\varphi) w_{j} \right ) \dx \dt, \\
\int_{0}^{T} \int_{\Omega} \delta(t) \mu w_{j} \dx \dt & = \int_{0}^{T} \int_{\Omega} \delta(t) \left ( A \Psi'(\varphi) w_{j} + B \nabla \varphi \cdot \nabla w_{j} - \chi_{\varphi} \sigma w_{j} \right ) \dx  \dt,  \\
\int_{0}^{T} \delta(t) \inner{\pd_{t}\sigma}{w_{j}} \dt & = \int_{0}^{T} \int_{\Omega} \delta(t) \left ( - n(\varphi) (\chi_{\sigma} \nabla \sigma - \chi_{\varphi} \nabla \varphi) \cdot \nabla w_{j} - \lambda_{c} \sigma h(\varphi) w_{j} \right ) \dx \dt \\
\notag & + \int_{0}^{T} \int_{\Gamma} \delta(t) K (\sigma_{\infty} - \sigma) w_{j} \dHaus \dt. 
\end{aligned}
\end{equation*}
Since this holds for all $\delta \in C^{\infty}_{c}(0,T)$, we infer that $(\varphi, \mu, \sigma)$ satisfies
\begin{subequations}\label{equ:limit}
\begin{align}
\inner{\pd_{t}\varphi}{w_{j}} & = \int_{\Omega} -m(\varphi) \nabla \mu \cdot \nabla w_{j} + (\lambda_{p} \sigma - \lambda_{a}) h(\varphi) w_{j}  \dx, \\
\int_{\Omega} \mu w_{j} \dx & = \int_{\Omega} A \Psi'(\varphi) w_{j} + B \nabla \varphi \cdot \nabla w_{j} - \chi_{\varphi} \sigma w_{j} \dx, \label{limit:mu} \\
\inner{\pd_{t}\sigma}{w_{j}} & = \int_{\Omega} -n(\varphi) (\chi_{\sigma} \nabla \sigma - \chi_{\varphi} \nabla \varphi) \cdot \nabla w_{j} - \lambda_{c} \sigma h(\varphi) w_{j} \dx \label{limit:sigma} \\
\notag & + \int_{\Gamma} K  (\sigma_{\infty} - \sigma) w_{j} \dHaus,
\end{align}
\end{subequations}
for a.e. $t \in (0,T)$ and for all $j \geq 1$.  As $\{ w_{j}\}_{j \in \N}$ is a basis for $H^{1}$, we see that the triplet $( \varphi, \mu, \sigma )$ satisfies \eqref{CHNutrient:truncated:weakform} for all $\zeta, \lambda, \xi \in H^{1}$.  Moreover, the strong convergence of $\varphi_{k}$ to $\varphi$ in $C([0,T];L^{2})$ and the fact that $\varphi_{k}(0) \to \varphi_{0}$ in $L^{2}$ imply that $\varphi(0) = \varphi_{0}$.  Similarly, by the continuous embedding
\begin{align*}
L^{2}(0,T;H^{1}) \cap H^{1}(0,T;(H^{1})^{*}) \subset C([0,T];L^{2}),
\end{align*}
and that $\sigma_{k}(0) \to \sigma_{0}$ in $L^{2}$, we have $\sigma(0) = \sigma_{0}$.  This shows that $(\varphi, \mu, \sigma)$ is a weak solution of \eqref{CHNutrient:truncated:weakform}.

\section{Continuous dependence}\label{sec:Uniqueness}
Suppose we have two weak solution triplets $\{\varphi_{i}, \mu_{i}, \sigma_{i}\}_{i = 1,2}$ to \eqref{Intro:CHNutrient} satisfying the assumptions of Theorem \ref{thm:ctsdep}.  Let us denote the differences by
\begin{align}
\varphi := \varphi_{1} - \varphi_{2}, \quad \sigma := \sigma_{1} - \sigma_{2}, \quad \mu := \mu_{1} - \mu_{2}, \quad \Sigma_{\infty} := \sigma_{\infty,1} - \sigma_{\infty,2}.
\end{align}
Then, we see that
\begin{align*}
\varphi & \in L^{\infty}(0,T;H^{1}) \cap H^{1}(0,T;(H^{1})^{*}), \quad \mu \in L^{2}(0,T;H^{1}), \\
\sigma & \in L^{2}(0,T;H^{1}) \cap L^{\infty}(0,T;L^{2}) \cap H^{1}(0,T;(H^{1})^{*}) \cap L^{2}(0,T;L^{2}(\Gamma))
\end{align*}
satisfy
\begin{subequations}
\begin{align}
\inner{\pd_{t} \varphi}{\zeta} & =  \int_{\Omega} - \nabla \mu \cdot \nabla \zeta + \lambda_{p} (\sigma_{1} h(\varphi_{1}) - \sigma_{2} h(\varphi_{2})) \zeta - \lambda_{a} (h(\varphi_{1}) - h(\varphi_{2})) \zeta \dx, \label{difference:varphi} \\
\int_{\Omega} \mu \lambda \dx & =  \int_{\Omega} A (\Psi'(\varphi_{1}) - \Psi'(\varphi_{2})) \lambda + B \nabla \varphi \cdot \nabla \lambda - \chi_{\varphi} \sigma \lambda \dx, \label{difference:mu} \\
\inner{\pd_{t} \sigma}{\xi} & = \int_{\Omega} -(\chi_{\sigma} \nabla \sigma - \chi_{\varphi} \nabla \varphi) \cdot \nabla \xi \dx \label{difference:sigma} \\
\notag & - \int_{\Omega} \lambda_{c} (\sigma_{1} h(\varphi_{1}) - \sigma_{2} h(\varphi_{2})) \xi \dx + \int_{\Gamma} K ( \Sigma_{\infty} - \sigma) \xi \dHaus, 
\end{align}
\end{subequations}
for all $\zeta, \lambda, \xi \in H^{1}$ and for a.e. $t \in (0,T)$.  Testing with $\zeta = \varphi$, $\xi = \sigma$, $\lambda = \mu - \chi_{\varphi} \sigma$ leads to
\begin{subequations}
\begin{align}
\frac{1}{2} \frac{\dd}{\dt} \norm{\varphi}_{L^{2}}^{2} & = \int_{\Omega} - \nabla \mu \cdot \nabla \varphi + \lambda_{p}(\sigma_{1} h(\varphi_{1}) - \sigma_{2} h(\varphi_{2})) \varphi - \lambda_{a} (h(\varphi_{1}) - h(\varphi_{2})) \varphi \dx, \label{uniq:apr1} \\
\frac{1}{2} \frac{\dd}{\dt} \norm{\sigma}_{L^{2}}^{2} & =  - \chi_{\sigma} \norm{\nabla \sigma}_{L^{2}}^{2} - K \norm{\sigma}_{L^{2}(\Gamma)}^{2} \label{uniq:apr2} \\
\notag & + \int_{\Omega} \chi_{\varphi} \nabla \varphi \cdot \nabla \sigma - \lambda_{c} (\sigma_{1} h(\varphi_{1}) - \sigma_{2} h(\varphi_{2})) \sigma \dx + K \int_{\Gamma} \Sigma_{\infty} \sigma \dHaus, \\
\norm{\mu}_{L^{2}}^{2}  & = \int_{\Omega} A(\Psi'(\varphi_{1}) - \Psi'(\varphi_{2})) (\mu - \chi_{\varphi} \sigma) + B \nabla \varphi \cdot \nabla (\mu - \chi_{\varphi} \sigma) \dx  \label{uniq:apr3} \\
\notag & + \chi_{\varphi}^{2} \norm{\sigma}_{L^{2}}^{2}.
\end{align}
\end{subequations}
Upon adding the products of $B$ with \eqref{uniq:apr1} and \eqref{uniq:apr2} with \eqref{uniq:apr3}, we obtain
\begin{equation}\label{uniq:apriori}
\begin{aligned}
&  \frac{B}{2} \frac{\dd}{\dt} \left ( \norm{\sigma}_{L^{2}}^{2} + \norm{\varphi}_{L^{2}}^{2} \right ) + \norm{\mu}_{L^{2}}^{2} - \chi_{\varphi}^{2} \norm{\sigma}_{L^{2}}^{2} + B\chi_{\sigma} \norm{\nabla \sigma}_{L^{2}}^{2} + BK \norm{\sigma}_{L^{2}(\Gamma)}^{2}\\
& =  \int_{\Omega} (\sigma_{1} h(\varphi_{1}) - \sigma_{2} h(\varphi_{2}))(\lambda_{p} B \varphi - \lambda_{c} B \sigma) + A(\Psi'(\varphi_{1}) - \Psi'(\varphi_{2}))(\mu - \chi_{\varphi} \sigma) \dx  \\
& -  B\lambda_{a} \int_{\Omega} (h(\varphi_{1}) - h(\varphi_{2})) \varphi \dx + BK \int_{\Gamma} \Sigma_{\infty} \sigma \dHaus \\
& \leq \int_{\Omega} (\abs{\sigma_{1}} \mathrm{L}_{h} \abs{\varphi} + h_{\infty} \abs{\sigma}) (\lambda_{p} B \abs{\varphi} + \lambda_{c} B \abs{\sigma}) + A \mathrm{L}_{\Psi'} \abs{\varphi}(\abs{\mu} + \chi_{\varphi} \abs{\sigma}) \dx \\
& + \int_{\Omega}  B \lambda_{a} \mathrm{L}_{h} \abs{\varphi}^{2} \dx + \frac{BK}{2} \norm{\Sigma_{\infty}}_{L^{2}(\Gamma)}^{2} + \frac{BK}{2} \norm{\sigma}_{L^{2}(\Gamma)}^{2},
\end{aligned}
\end{equation}
where we have used H\"{o}lder's inequality and Young's inequality on the boundary term involving $\Sigma_{\infty}$ and the Lipschitz assumptions on $h(\cdot)$ and $\Psi'(\cdot)$ to deduce that
\begin{align*}
\abs{\sigma_{1} h(\varphi_{1}) - \sigma_{2} h(\varphi_{2}))} & \leq \abs{\sigma_{1}} \abs{h(\varphi_{1}) - h(\varphi_{2})} + \abs{\sigma} \abs{h(\varphi_{2})} \leq \abs{\sigma_{1}} \mathrm{L}_{h} \abs{\varphi} + h_{\infty} \abs{\sigma}, \\
\abs{\Psi'(\varphi_{1}) - \Psi'(\varphi_{2})} & \leq \mathrm{L}_{\Psi'} \abs{\varphi}.
\end{align*}
Next, let us consider a constant $\mathcal{X} > 0$, yet to be determined, and consider testing with $\lambda = \mathcal{X} \varphi$ in \eqref{difference:mu}.  Then H\"{o}lder's inequality and Young's inequality lead to
\begin{equation}\label{nablavarphiL2:part}
\begin{aligned}
&  B \mathcal{X} \norm{\nabla \varphi}_{L^{2}}^{2}  = \mathcal{X} \int_{\Omega} (\mu + A (\Psi'(\varphi_{2}) - \Psi'(\varphi_{1})) + \chi_{\varphi} \sigma) \varphi \dx \\
& \quad \leq C(\mathcal{X}, A, \chi_{\varphi}, \mathrm{L}_{\Psi'}) \left ( \norm{\mu}_{L^{2}} \norm{\varphi}_{L^{2}} + \norm{\varphi}_{L^{2}}^{2} + \norm{\sigma}_{L^{2}} \norm{\varphi}_{L^{2}} \right ) \\
& \quad \leq \frac{1}{4} \norm{\mu}_{L^{2}}^{2} + C(\mathcal{X}, A, \chi_{\varphi}, \mathrm{L}_{\Psi'}) \left ( \norm{\varphi}_{L^{2}}^{2} + \norm{\sigma}_{L^{2}}^{2} \right ).
\end{aligned}
\end{equation}
Adding \eqref{nablavarphiL2:part} to \eqref{uniq:apriori} yields that
\begin{equation}\label{uniq:apriori:version2}
\begin{aligned}
& \frac{B}{2} \frac{\dd}{\dt} \left ( \norm{\sigma}_{L^{2}}^{2} + \norm{\varphi}_{L^{2}}^{2} \right ) - \chi_{\varphi}^{2} \norm{\sigma}_{L^{2}}^{2} \\
& \quad + B\chi_{\sigma} \norm{\nabla \sigma}_{L^{2}}^{2} + \frac{BK}{2} \norm{\sigma}_{L^{2}(\Gamma)}^{2} + B \mathcal{X} \norm{\nabla \varphi}_{L^{2}}^{2} + \norm{\mu}_{L^{2}}^{2}  \\
& \quad \leq \frac{1}{4} \norm{\mu}_{L^{2}}^{2} + C(\mathcal{X}, A, B, \mathrm{L}_{h}, \lambda_{a},  \chi_{\varphi}, \mathrm{L}_{\Psi'}) \left ( \norm{\varphi}_{L^{2}}^{2} + \norm{\sigma}_{L^{2}}^{2} \right ) + \frac{BK}{2} \norm{\Sigma_{\infty}}_{L^{2}(\Gamma)}^{2} \\
& \quad +  \int_{\Omega} (\abs{\sigma_{1}} \mathrm{L}_{h} \abs{\varphi} + h_{\infty} \abs{\sigma}) (\lambda_{p} B \abs{\varphi} + \lambda_{c} B \abs{\sigma}) + A \mathrm{L}_{\Psi'} \abs{\varphi}(\abs{\mu} + \chi_{\varphi} \abs{\sigma}) \dx .
\end{aligned}
\end{equation}
By H\"{o}lder's inequality, Young's inequality and the following Sobolev embedding for dimensions $d \leq 4$,
\begin{align}\label{Sobo:H1L4}
\norm{f}_{L^{4}} \leq C_{\mathrm{S}} \norm{f}_{H^{1}} \quad \forall f \in H^{1},
\end{align} 
where $C_{\mathrm{S}}$ is a positive constant depending only on $\Omega$ and $d$, we have
\begin{align*}
& \int_{\Omega} (\abs{\sigma_{1}} \mathrm{L}_{h} \abs{\varphi} + h_{\infty} \abs{\sigma}) (\lambda_{p} B \abs{\varphi} + \lambda_{c} B \abs{\sigma}) + A \mathrm{L}_{\Psi'} \abs{\varphi}(\abs{\mu} + \chi_{\varphi} \abs{\sigma}) \dx \\
& \quad \leq \mathrm{L}_{h} \lambda_{p} B \norm{\sigma_{1}}_{L^{2}} \norm{\varphi}_{L^{4}}^{2} + \mathrm{L}_{h} \lambda_{c} B \norm{\sigma_{1}}_{L^{2}} \norm{\varphi}_{L^{4}} \norm{\sigma}_{L^{4}} \\
& \quad + C(\lambda_{p}, B, \lambda_{c}, h_{\infty}, A, \mathrm{L}_{\Psi'}, \chi_{\varphi}) \left ( \norm{\varphi}_{L^{2}}^{2} + \norm{\sigma}_{L^{2}}^{2} \right ) + \frac{1}{4} \norm{\mu}_{L^{2}}^{2} \\
& \quad \leq  \left ( C_{\mathrm{S}}^{2} \mathrm{L}_{h} B\lambda_{p} \norm{\sigma_{1}}_{L^{\infty}(0,T;L^{2})} + \frac{B}{2 \chi_{\sigma}} C_{\mathrm{S}}^{4} \mathrm{L}_{h}^{2} \lambda_{c}^{2} \norm{\sigma_{1}}_{L^{\infty}(0,T;L^{2})}^{2} \right ) \norm{\varphi}_{H^{1}}^{2} \\
& \quad + C\left ( \norm{\varphi}_{L^{2}}^{2} + \norm{\sigma}_{L^{2}}^{2} \right ) + \frac{1}{4} \norm{\mu}_{L^{2}}^{2} + \frac{B \chi_{\sigma}}{2} \norm{\nabla \sigma}_{L^{2}}^{2},
\end{align*}
where the positive constant $C$ depends on $\lambda_{p}$, $B$, $\lambda_{c}$, $h_{\infty}$, $A$, $\mathrm{L}_{\Psi'}$, $\chi_{\varphi}$ and $\chi_{\sigma}$.  In turn, from \eqref{uniq:apriori:version2} we obtain
\begin{equation}
\begin{aligned}
& \frac{B}{2}  \frac{\dd}{\dt} \left ( \norm{\sigma}_{L^{2}}^{2} + \norm{\varphi}_{L^{2}}^{2} \right )\\
& \quad + \frac{1}{2}\norm{\mu}_{L^{2}}^{2} + \frac{B \chi_{\sigma}}{2} \norm{\nabla \sigma}_{L^{2}}^{2} + \frac{BK}{2} \norm{\sigma}_{L^{2}(\Gamma)}^{2} + B \mathcal{X} \norm{\nabla \varphi}_{L^{2}}^{2} \\
& \quad - B C_{\mathrm{S}}^{2} \mathrm{L}_{h} \norm{\sigma_{1}}_{L^{\infty}(0,T;L^{2})} \left ( \lambda_{p} + \frac{1}{2 \chi_{\sigma}} C_{\mathrm{S}}^{2} \mathrm{L}_{h} \lambda_{c}^{2} \norm{\sigma_{1}}_{L^{\infty}(0,T;L^{2})} \right ) \norm{\nabla \varphi}_{L^{2}}^{2}  \\
& \quad \leq C \left ( \norm{\varphi}_{L^{2}}^{2} + \norm{\sigma}_{L^{2}}^{2} \right ) + \frac{BK}{2} \norm{\Sigma_{\infty}}_{L^{2}(\Gamma)}^{2},
\end{aligned}
\end{equation}
where the constant $C$ depends on $\norm{\sigma_{1}}_{L^{\infty}(0,T;L^{2})}$, $C_{\mathrm{S}}$, $A$, $B$, $\mathrm{L}_{h}$, $\lambda_{p}$, $\lambda_{c}$, $h_{\infty}$, $\chi_{\varphi}$, $\chi_{\sigma}$, $\mathcal{X}$, and $\mathrm{L}_{\Psi'}$.  We now choose
\begin{align*}
\mathcal{X} > \left ( C_{\mathrm{S}}^{2} \mathrm{L}_{h} \lambda_{p} \norm{\sigma_{1}}_{L^{\infty}(0,T;L^{2})} + \frac{1}{2 \chi_{\sigma}} C_{\mathrm{S}}^{4} \mathrm{L}_{h}^{2} \lambda_{c}^{2} \norm{\sigma_{1}}_{L^{\infty}(0,T;L^{2})}^{2} \right ),
\end{align*}
and so there exist constants $c, C > 0$ such that
\begin{align*}
&  \frac{\dd}{\dt} \left ( \norm{\sigma}_{L^{2}}^{2} + \norm{\varphi}_{L^{2}}^{2} \right ) - C \left ( \norm{\sigma}_{L^{2}}^{2} + \norm{\varphi}_{L^{2}}^{2} \right ) \\
& \quad +  \norm{\mu}_{L^{2}}^{2} + \norm{\nabla \sigma}_{L^{2}}^{2} + \norm{\sigma}_{L^{2}(\Gamma)}^{2} + \norm{\nabla \varphi}_{L^{2}}^{2} \leq c \norm{\Sigma_{\infty}}_{L^{2}(\Gamma)}^{2},
\end{align*}
and a Gronwall argument yields 
\begin{align*}
& \left ( \norm{\sigma(s)}_{L^{2}}^{2} + \norm{\varphi(s)}_{L^{2}}^{2} \right ) + \int_{0}^{s} \norm{\mu}_{L^{2}}^{2} + \norm{\nabla \sigma}_{L^{2}}^{2} + \norm{\sigma}_{L^{2}(\Gamma)}^{2} + \norm{\nabla \varphi}_{L^{2}}^{2} \dt \\
& \quad \leq c \exp(CT) \norm{\Sigma_{\infty}}_{L^{2}(0,T;L^{2}(\Gamma))}^{2} + \exp(CT) \left ( \norm{\sigma(0)}_{L^{2}}^{2} + \norm{\varphi(0)}_{L^{2}}^{2} \right )
\end{align*}
for any $s \in (0,T]$.  Taking the supremum in $s$ on the left-hand side yields the desired result.

\section{Quasi-static nutrient}\label{sec:quasi}
For the existence of weak solutions to \eqref{Intro:CHN:quasi}, we will only show the existence of solutions at the level of the Galerkin approximation and provide the necessary a priori estimates.  

\subsection{Existence of Galerkin solutions}\label{sec:quasi:Galerkin}
Similar to Section \ref{sec:Galerkinapprox}, we consider the Galerkin ansatz
\begin{subequations}\label{quasi:Galerkin:system}
\begin{align}
\int_{\Omega} \pd_{t} \varphi_{k} w_{j} \dx & = \int_{\Omega} -m(\varphi_{k}) \nabla \mu_{k} \cdot \nabla w_{j} + (\lambda_{p} \sigma_{k} - \lambda_{a}) h(\varphi_{k}) w_{j} \dx, \label{quasi:Galerkin:varphi} \\
\int_{\Omega} \mu_{k} w_{j} \dx & = \int_{\Omega} A \Psi'(\varphi_{k}) w_{j} + B \nabla \varphi_{k} \cdot \nabla w_{j} - \chi_{\varphi} \sigma_{k} w_{j} \dx,  \label{quasi:Galerkin:mu} \\
\int_{\Gamma} K ( \sigma_{\infty} - \sigma_{k}) w_{j} \dHaus & = \int_{\Omega} D(\varphi_{k}) ( \nabla \sigma_{k} - \eta \nabla \varphi_{k}) \cdot \nabla w_{j} + \lambda_{c} \sigma_{k} h(\varphi_{k}) w_{j} \dx, \label{quasi:Galerkin:sigma} 
\end{align}
\end{subequations}
with the finite-dimensional functions $\varphi_{k}$, $\sigma_{k}$ and $\mu_{k}$ as defined in \eqref{Galerkin:ansatz}.  Then, \eqref{quasi:Galerkin:system} can be written in terms of the following initial value problem
\begin{subequations}\label{quasi:discrete:system}
\begin{align}
\frac{\dd}{\dt} \bm{\alpha}_{k} & = - \bm{S}_{m}^{k} \bm{\beta}_{k} + \lambda_{p} \bm{M}_{h}^{k} \bm{\gamma}_{k} - \lambda_{a} \bm{h}^{k}, \label{quasi:discrete:varphi} \\
 \bm{\beta}_{k} & = A \bm{\psi}^{k} + B\bm{S} \bm{\alpha}_{k} - \chi_{\varphi} \bm{\gamma}_{k}, \label{quasi:discrete:mu} \\
\bm{0} & =  \bm{S}_{D}^{k}( \bm{\gamma}_{k} - \eta \bm{\alpha}_{k}) + \lambda_{c} \bm{M}_{h}^{k} \bm{\gamma}_{k} + K\bm{M}_{\Gamma} \bm{\gamma}_{k} - K \bm{\Sigma}^{k}, \label{quasi:discrete:sigma}
\end{align}
\end{subequations}
with initial data $\bm{\alpha}_{k}(0)$ defined in \eqref{discrete:system:combined:initial}.  Here, the matrix $\bm{S}_{D}^{k}$ is defined as
\begin{align*}
(\bm{S}_{D}^{k})_{ji} & := \int_{\Omega} D(\varphi_{k}) \nabla w_{i} \cdot \nabla w_{j} \dx, 
\end{align*}
for $1 \leq i, j \leq k$.  Upon rearranging, we see that \eqref{quasi:discrete:sigma} can be written as
\begin{align*}
( \bm{S}_{D}^{k} + \lambda_{c} \bm{M}_{h}^{k} + K \bm{M}_{\Gamma}) \bm{\gamma}_{k} = \eta \bm{S}_{D}^{k}  \bm{\alpha}_{k} + K \bm{\Sigma}^{k}.
\end{align*} 
Note that for a general coefficient vector $\bm{\xi} = (\xi_{1}, \dots, \xi_{k})^{\top} \in \R^{k}$ corresponding to $v := \sum_{i=1}^{k} \xi_{i} w_{i} \in W_{k}$, we have
\begin{align*}
\bm{\xi}^{\top} (\bm{S}_{D}^{k} + \lambda_{c} \bm{M}_{h}^{k} + K \bm{M}_{\Gamma}) \bm{\xi}  =\int_{\Omega} D(\varphi_{k}) \abs{\nabla v}^{2} + \lambda_{c} h(\varphi_{k}) \abs{v}^{2} \dx + \int_{\Gamma} K \abs{v}^{2} \dx  \geq 0,
\end{align*}
where we used that $\lambda_{c} \geq 0$, $h(\cdot) \geq 0$ and $D(\cdot) > 0$.  This in turn implies that $\bm{S}_{D}^{k} + \lambda_{c} \bm{M}_{h}^{k} + K \bm{M}_{\Gamma}$ is positive semi-definite.  Moreover, by the Poincar\'{e} inequality \eqref{boundary:Poincare} it is clear that
\begin{align*}
0 = \bm{\xi}^{\top} (\bm{S}_{D}^{k} + \lambda_{c} \bm{M}_{h}^{k} + K \bm{M}_{\Gamma}) \bm{\xi} \Longleftrightarrow v = 0 \Longleftrightarrow \bm{\xi} = \bm{0},
\end{align*}
and thus $\bm{S}_{D}^{k} + \lambda_{c} \bm{M}_{h}^{k} + K \bm{M}_{\Gamma}$ is an invertible positive definite matrix.  We can now write \eqref{quasi:discrete:system} in terms of an initial value problem in $\bm{\alpha}_{k}$,
\begin{equation}\label{quasi:ODE}
\begin{aligned}
\frac{\dd}{\dt} \bm{\alpha}_{k} & = - B \bm{S}_{m}^{k}  \bm{S}  \bm{\alpha}_{k} - \lambda_{a}  \bm{h}^{k} -  A \bm{S}_{m}^{k} \bm{\psi}^{k} \\
& + (\chi_{\varphi} \bm{S}_{m}^{k}  + \lambda_{p} \bm{M}_{h}^{k})(\bm{S}_{D}^{k} + \lambda_{c} \bm{M}_{h}^{k} + K \bm{M}_{\Gamma})^{-1} ( \eta \bm{S}_{D} \bm{\alpha}_{k} + K \bm{\Sigma}^{k}),
\end{aligned}
\end{equation}
with $\bm{\alpha}_{k}(0)$ as defined in \eqref{discrete:system:combined:initial}.  We find that the right-hand side of \eqref{quasi:ODE} depends continuously on $\bm{\alpha}_{k}$, and for every $k \in \N$ the existence of a local solution defined on $[0,t_{k}]$ is guaranteed by the Cauchy--Peano theorem.  

\subsection{A priori estimates}
The derivation of a priori estimates for the Galerkin solutions follows in a similar manner to Section \ref{sec:Galerkinapprox}.  Multiplying \eqref{quasi:Galerkin:varphi} with $\beta_{j}^{k} + \chi_{\varphi} \gamma_{j}^{k}$ and \eqref{quasi:Galerkin:mu} with $\frac{\dd}{\dt} \alpha_{j}^{k}$, and summing from $j = 1$ to $k$ gives
\begin{equation}\label{quasi:apriori:firstpart}
\begin{aligned}
& \frac{\dd}{\dt} \int_{\Omega} \left [ A \Psi'(\varphi_{k}) + \frac{B}{2} \abs{\nabla \varphi_{k}}^{2} \right ] \dx + \int_{\Omega} m(\varphi_{k}) \nabla \mu_{k} \cdot \nabla (\mu_{k} + \chi_{\varphi} \sigma_{k}) \dx \\
& \quad = \int_{\Omega} (\lambda_{p} \sigma_{k} - \lambda_{a}) h(\varphi_{k}) (\mu_{k} + \chi_{\varphi} \sigma_{k}) \dx.
\end{aligned}
\end{equation}
Let $\mathcal{W}$ denote a positive constant yet to be determined.  We multiply \eqref{quasi:Galerkin:sigma} with $\mathcal{W} \gamma_{j}^{k}$ and sum from $j = 1$ to $k$, leading to
\begin{equation}
\label{quasi:apriori:secondpart}
\begin{aligned}
& \mathcal{W} \int_{\Omega} D(\varphi_{k}) ( \abs{\nabla \sigma_{k}}^{2} - \eta \nabla \varphi_{k} \cdot \nabla \sigma_{k}) + \lambda_{c} \abs{\sigma_{k}}^{2} h(\varphi_{k}) \dx \\
& \quad = \int_{\Gamma} \mathcal{W} K (\sigma_{\infty} - \sigma_{k}) \sigma_{k} \dHaus.
\end{aligned}
\end{equation}
Summing \eqref{quasi:apriori:firstpart} and \eqref{quasi:apriori:secondpart} leads to
\begin{equation}\label{quasi:apriori:energy}
\begin{aligned}
& \frac{\dd}{\dt} \int_{\Omega} \left [ A \Psi(\varphi_{k}) + \frac{B}{2} \abs{\nabla \varphi_{k}}^{2} \right ] \dx + \int_{\Gamma} \mathcal{W}K  \abs{\sigma_{k}}^{2} \dHaus  \\
& \quad + \int_{\Omega} m(\varphi_{k}) \abs{\nabla \mu_{k}}^{2} + \mathcal{W} D(\varphi_{k}) \abs{\nabla \sigma_{k}}^{2} + \mathcal{W} \lambda_{c} h(\varphi_{k}) \abs{\sigma_{k}}^{2} \dx \\
& \quad = \int_{\Omega} (\lambda_{p} \sigma_{k} - \lambda_{a}) h(\varphi_{k}) (\mu_{k} + \chi_{\varphi} \sigma_{k}) - \chi_{\varphi} m(\varphi_{k}) \nabla \mu_{k} \cdot \nabla \sigma_{k} \dx \\
& \quad + \int_{\Omega} \mathcal{W} D(\varphi_{k}) \eta \nabla \varphi_{k} \cdot \nabla \sigma_{k} \dx +\int_{\Gamma} \mathcal{W} K \sigma_{\infty} \sigma_{k} \dHaus.
\end{aligned}
\end{equation}
Neglecting the non-negative term $\int_{\Omega} \lambda_{c}(\varphi_{k}) \abs{\sigma_{k}}^{2} \dx$, and using the boundedness of $m(\cdot)$, $D(\cdot)$, and $h(\cdot)$, and applying H\"{o}lder's inequality and Young's inequality we have
\begin{equation}
\begin{aligned}
& \frac{\dd}{\dt} \int_{\Omega} \left [ A \Psi(\varphi_{k}) + \frac{B}{2} \abs{\nabla \varphi_{k}}^{2} \right ] \dx + \frac{m_{0}}{2} \norm{\nabla \mu_{k}}_{L^{2}}^{2} + \mathcal{W} \frac{D_{0}}{2} \norm{\nabla \sigma_{k}}_{L^{2}}^{2}  +  \mathcal{W} \frac{K}{2} \norm{\sigma_{k}}_{L^{2}(\Gamma)}^{2} \\
& \quad \leq \mathcal{W} \frac{K}{2} \norm{\sigma_{\infty}}_{L^{2}(\Gamma)}^{2} + \frac{\chi_{\varphi}^{2} m_{1}}{2} \norm{\nabla \sigma_{k}}_{L^{2}}^{2} + \mathcal{W} \frac{D_{1} \eta^{2}}{2} \norm{\nabla \varphi_{k}}_{L^{2}}^{2} \\
& \quad + \left ( h_{\infty} \lambda_{p} \frac{d_{1}}{2} + d_{2} \right ) \norm{\mu_{k}}_{L^{2}}^{2}  + h_{\infty} \left (  \lambda_{p} \frac{1}{2 d_{1}} + \lambda_{p} \chi_{\varphi} + d_{3} \right ) \norm{\sigma_{k}}_{L^{2}}^{2} \\
& \quad + C(d_{2},  d_{3}, \chi_{\varphi}, \lambda_{a}, h_{\infty}, \abs{\Omega}),
\end{aligned}
\end{equation}
for some positive constants $d_{1}, d_{2}, d_{3}$ yet to be determined.  Employing \eqref{boundary:Poincare}, we see that
\begin{align}\label{boundary:Poincare:sigma}
\norm{\sigma_{k}}_{L^{2}}^{2} \leq 2 C_{\mathrm{P}}^{2} \left ( \norm{\nabla \sigma_{k}}_{L^{2}}^{2} + \norm{\sigma_{k}}_{L^{2}(\Gamma)}^{2} \right ),
\end{align}
and from \eqref{muL2estimate} and \eqref{varphiL2PsiL1} we have
\begin{equation}\label{muL2estimate:Psi}
\begin{aligned}
\norm{\mu_{k}}_{L^{2}}^{2} &  \leq 2 C_{\mathrm{P}}^{2} \norm{\nabla \mu_{k} }_{L^{2}}^{2} + 6 A^{2} R_{3}^{2} \norm{\varphi_{k}}_{L^{2}}^{2} + 6 \chi_{\varphi}^{2} \norm{\sigma_{k}}_{L^{2}}^{2} + C(A, R_{3}, \abs{\Omega}) \\
& \leq 2C_{\mathrm{P}}^{2} \norm{\nabla \mu_{k}}_{L^{2}}^{2} + \frac{6 A^{2} R_{3}^{2}}{R_{1}} \norm{\Psi(\varphi_{k})}_{L^{1}} + 6 \chi_{\varphi}^{2} \norm{\sigma_{k}}_{L^{2}}^{2} \\
& + C(A, R_{1}, R_{2},  R_{3}, \abs{\Omega}).
\end{aligned}
\end{equation}
Substituting \eqref{boundary:Poincare:sigma} and \eqref{muL2estimate:Psi} into \eqref{quasi:apriori:energy} leads to
\begin{equation}
\begin{aligned}
&\frac{\dd}{\dt} \left [ A \norm{\Psi(\varphi_{k})}_{L^{1}} + \frac{B}{2} \norm{\nabla \varphi_{k}}_{L^{2}}^{2} \right ] - \mathcal{W} \frac{D_{1} \eta^{2}}{2} \norm{\nabla \varphi_{k}}_{L^{2}}^{2} \\
& + \norm{\nabla \mu_{k}}_{L^{2}}^{2} \left ( \frac{m_{0}}{2} - 2C_{\mathrm{P}}^{2} \left (h_{\infty} \lambda_{p} \frac{d_{1}}{2} + d_{2} \right ) \right ) \\
& + \norm{\sigma_{k}}_{L^{2}(\Gamma)}^{2} \left ( \mathcal{W} \frac{K}{2} - 2 C_{\mathrm{P}}^{2} \left ( h_{\infty} \left (  \lambda_{p} \frac{1}{2 d_{1}} + \lambda_{p} \chi_{\varphi} + d_{3} \right ) + 6 \chi_{\varphi}^{2} \left (h_{\infty} \lambda_{p} \frac{d_{1}}{2} + d_{2} \right ) \right ) \right ) \\
& + \norm{\nabla \sigma_{k}}_{L^{2}}^{2} \left ( \mathcal{W} \frac{D_{0}}{2} - \frac{\chi_{\varphi}^{2} m_{1}}{2} - 2C_{\mathrm{P}}^{2} \left ( h_{\infty} \left (  \lambda_{p} \frac{1}{2 d_{1}} + \lambda_{p} \chi_{\varphi} + d_{3} \right ) + 6 \chi_{\varphi}^{2} \left (h_{\infty} \lambda_{p} \frac{d_{1}}{2} + d_{2} \right ) \right )\right ) \\
& -\norm{\Psi(\varphi_{k})}_{L^{1}} \frac{6 A^{2} R_{3}^{2}}{R_{1}}\left (h_{\infty} \lambda_{p} \frac{d_{1}}{2} + d_{2} \right ) \\
& \leq C(R_{1}, R_{2}, R_{3}, A, \mathcal{W}, K, d_{2}, d_{3}, \chi_{\varphi}, \lambda_{a}, h_{\infty}, \abs{\Omega}) \left (1 + \norm{\sigma_{\infty}}_{L^{2}(\Gamma)}^{2} \right ).
\end{aligned}
\end{equation}
We choose
\begin{align*}
d_{1} = \frac{m_{0}}{8 h_{\infty} \lambda_{p} C_{\mathrm{P}}^{2}} , \quad d_{2} = \frac{m_{0}}{16 C_{\mathrm{P}}^{2}}, \quad d_{3} = 1,
\end{align*}
and
\begin{align*}
\mathcal{W} > \min \left ( \frac{2}{K}, \frac{2}{D_{0}} \right ) \left ( \frac{\chi_{\varphi}^{2} m_{1}}{2} +\frac{3}{4} m_{0} \chi_{\varphi}^{2} + 2C_{\mathrm{P}}^{2} h_{\infty} \left ( \frac{4\lambda_{p}^{2} h_{\infty} C_{\mathrm{P}}^{2}}{m_{0}} + \lambda_{p} \chi_{\varphi} + 1 \right )  \right )
\end{align*}
so that there exists a positive constant $\overline{c}$ such that
\begin{align*}
& \frac{\dd}{\dt} \left [ A \norm{\Psi(\varphi_{k})}_{L^{1}} + \frac{B}{2} \norm{\nabla \varphi_{k}}_{L^{2}}^{2} \right ] - \frac{3 A^{2} R_{3}^{2} m_{0}}{4 C_{\mathrm{P}}^{2} R_{1}} \norm{\Psi(\varphi_{k})}_{L^{1}} - \frac{\mathcal{W} D_{1} \eta^{2}}{2} \norm{\nabla \varphi_{k}}_{L^{2}}^{2} \\
& \quad + \overline{c} \left ( \norm{\nabla \mu_{k}}_{L^{2}}^{2} + \norm{\nabla \sigma_{k}}_{L^{2}}^{2} + \norm{\sigma_{k}}_{L^{2}(\Gamma)}^{2} \right  ) \leq C \left ( 1 + \norm{\sigma_{\infty}}_{L^{2}(\Gamma)}^{2} \right ).
\end{align*}
A Gronwall argument gives
\begin{equation}\label{quasi:apriori:first}
\begin{aligned}
\sup_{s \in (0,T]} & \left ( \norm{\Psi(\varphi_{k}(s))}_{L^{1}} + \norm{\nabla \varphi_{k}(s)}_{L^{2}}^{2} \right ) \\
& +  \norm{\nabla \mu_{k}}_{L^{2}(0,T;L^{2})}^{2} + \norm{\nabla \sigma_{k}}_{L^{2}(0,T;L^{2})}^{2} + \norm{\sigma_{k}}_{L^{2}(0,T;L^{2}(\Gamma))}^{2} \\
& \leq C \left ( 1 + \norm{\sigma_{\infty}}_{L^{2}(0,T;L^{2}(\Gamma))}^{2} \right ) 
\end{aligned}
\end{equation}
for some positive constant $C$ that does not depend on $\varphi_{k}$, $\sigma_{k}$ and $\mu_{k}$.  Here we see that for the quasi-static model \eqref{Intro:CHN:quasi} the assumption \eqref{assump:constantsrelations} for the constant $A$ is not used.  Invoking \eqref{boundary:Poincare:sigma} and \eqref{muL2estimate:Psi} give
\begin{align}\label{quasi:apriori:second}
\norm{\mu_{k}}_{L^{2}(0,T;L^{2})}^{2} + \norm{\sigma_{k}}_{L^{2}(0,T;L^{2})}^{2} \leq C \left (1 + \norm{\sigma_{\infty}}_{L^{2}(0,T;L^{2}(\Gamma))}^{2} \right ).
\end{align}
The above a priori estimates \eqref{quasi:apriori:first} and \eqref{quasi:apriori:second} imply that we can extend the solution $\{ \varphi_{k}, \mu_{k}, \sigma_{k}\}$ to the interval $[0,T]$, and thus $t_{k} = T$ for all $k \in \N$.  Together with \eqref{pdtvarphik:apriori} we obtain 
\begin{align*}
\{\varphi_{k}\}_{k \in \N} & \text{ bounded in } L^{\infty}(0,T;H^{1}) \cap H^{1}(0,T;(H^{1})^{*}), \\
\{\mu_{k}\}_{k \in \N } & \text{ bounded in } L^{2}(0,T;H^{1}), \\
 \{\sigma_{k}\}_{k \in \N} & \text{ bounded in } L^{2}(0,T;H^{1}) \cap L^{2}(0,T;L^{2}(\Gamma)).
\end{align*}
Uniform boundedness in the above spaces and the standard compactness arguments allow us to pass to the limit $k \to \infty$ in \eqref{quasi:Galerkin:system} to deduce the existence of a weak solution $(\varphi, \mu, \sigma)$ to \eqref{Intro:CHN:quasi} in the sense of Definition \ref{defn:weaksolutionquasi}.

\subsection{Further regularity}
Suppose that $\sigma_{\infty} \in L^{\infty}(0,T;L^{2}(\Gamma))$, then substituting $\xi = \sigma$ in \eqref{quasi:sigma:weak} leads to
\begin{align*}
\int_{\Gamma} K (\sigma_{\infty} - \sigma) \sigma \dHaus = \int_{\Omega} D(\varphi) (\nabla \sigma - \eta \nabla \varphi) \cdot \nabla \sigma + \lambda_{c} \abs{\sigma}^{2} h(\varphi) \dx.
\end{align*}
By the non-negativity of $\lambda_{c}$ and $h(\cdot)$, the boundedness of $D(\cdot)$, H\"{o}lder's inequality and Young's inequality, we obtain
\begin{align}\label{sigmaLinftyH1}
\frac{D_{0}}{2} \norm{\nabla \sigma}_{L^{2}}^{2} + \frac{K}{2} \norm{\sigma}_{L^{2}(\Gamma)}^{2} \leq \frac{K}{2} \norm{\sigma_{\infty}}_{L^{2}(\Gamma)}^{2} + \frac{D_{1} \eta^{2}}{2} \norm{\nabla \varphi}_{L^{2}}^{2}.
\end{align}
As $\varphi \in L^{\infty}(0,T;H^{1})$ and $\sigma_{\infty} \in L^{\infty}(0,T;L^{2}(\Gamma))$, taking the supremum of $t \in (0,T]$ in \eqref{sigmaLinftyH1} and by applying the Poincar\'{e} inequality \eqref{boundary:Poincare}, we find that
\begin{align*}
\sigma \in L^{\infty}(0,T;H^{1}).
\end{align*}

\subsection{Continuous dependence}\label{sec:quasi:ctsdep}
Suppose we have two weak solution triplets $\{ \varphi_{i}, \mu_{i}, \sigma_{i}\}_{i = 1,2}$ to \eqref{CHNutrient:quasi:weakform} satisfying the assumptions of Theorem \ref{thm:quasi:ctsdep}.  Let $\varphi$, $\mu$ and $\sigma$ denote the differences respectively.  Then 
\begin{align*}
\varphi & \in L^{\infty}(0,T;H^{1}) \cap H^{1}(0,T;(H^{1})^{*}), \\
\mu & \in L^{2}(0,T;H^{1}), \quad \sigma \in L^{\infty}(0,T;H^{1}),
\end{align*}
and
\begin{subequations}
\begin{align}
\inner{\pd_{t} \varphi}{\zeta} & =  \int_{\Omega} - \nabla \mu \cdot \nabla \zeta + \lambda_{p} (\sigma_{1} h(\varphi_{1}) - \sigma_{2} h(\varphi_{2})) \zeta - \lambda_{a} (h(\varphi_{1}) - h(\varphi_{2})) \zeta \dx, \label{quasi:difference:varphi} \\
\int_{\Omega} \mu \lambda  \dx & =  \int_{\Omega} A (\Psi'(\varphi_{1}) - \Psi'(\varphi_{2})) \lambda + B \nabla \varphi \cdot \nabla \lambda - \chi_{\varphi} \sigma \lambda \dx, \label{quasi:difference:mu} \\
0 & = \int_{\Omega} D (\nabla \sigma - \eta \nabla \varphi) \cdot \nabla \xi \dx + \int_{\Omega} \lambda_{c} (\sigma_{1} h(\varphi_{1}) - \sigma_{2} h(\varphi_{2})) \xi \dx \label{quasi:difference:sigma} \\
\notag & + \int_{\Gamma} K (\sigma - \Sigma_{\infty}) \xi \dHaus ,   
\end{align}
\end{subequations}
for all $\zeta, \lambda, \xi \in H^{1}$ and for a.e. $t \in (0,T)$.  Testing with $\zeta = \varphi$, $\xi = \sigma$, $\lambda = \varphi$, and $\lambda = \mu$ leads to
\begin{subequations}
\begin{align}
\frac{1}{2} \frac{\dd}{\dt} \norm{\varphi}_{L^{2}}^{2} & = \int_{\Omega} - \nabla \mu \cdot \nabla \varphi + \lambda_{p}(\sigma_{1}h(\varphi_{1}) - \sigma_{2} h(\varphi_{2})) \varphi - \lambda_{a} (h(\varphi_{1}) - h(\varphi_{2})) \varphi \dx , \label{quasi:uniq:apr1} \\
D \norm{\nabla \sigma}_{L^{2}}^{2} & =  - K \norm{\sigma}_{L^{2}(\Gamma)}^{2} + \int_{\Omega} D \eta \nabla \varphi \cdot \nabla \sigma - \lambda_{c} (\sigma_{1} h(\varphi_{1}) - \sigma_{2} h(\varphi_{2})) \sigma \dx \label{quasi:uniq:apr2} \\
\notag & + K \int_{\Gamma} \Sigma_{\infty} \sigma \dHaus, \\
\int_{\Omega} \mu \varphi \dx & = \int_{\Omega} A(\Psi'(\varphi_{1}) - \Psi'(\varphi_{2})) \varphi - \chi_{\varphi} \sigma \varphi \dx + B \norm{\nabla \varphi}_{L^{2}}^{2}, \label{quasi:uniq:apr3} \\
\norm{\mu}_{L^{2}}^{2} & = \int_{\Omega} A(\Psi'(\varphi_{1}) - \Psi'(\varphi_{2})) \mu + B \nabla \varphi \cdot \nabla \mu - \chi_{\varphi} \sigma \mu \dx. \label{quasi:uniq:apr4}  
\end{align}
\end{subequations}
We proceed similarly to Section \ref{sec:Uniqueness}.  Let $\mathcal{Y}$, $\mathcal{Z}$ denote two positive constants yet to be determined.  Upon adding the product of $B$ with \eqref{quasi:uniq:apr1}, the product of $\mathcal{Z}$ with \eqref{quasi:uniq:apr2}, the product of $\mathcal{Y}$ with \eqref{quasi:uniq:apr3}, and \eqref{quasi:uniq:apr4}, we obtain
\begin{equation}\label{uniq:quasi:energyidentity}
\begin{aligned}
& \frac{B}{2} \frac{\dd}{\dt} \norm{\varphi}_{L^{2}}^{2} + \norm{\mu}_{L^{2}}^{2}  + B \mathcal{Y} \norm{\nabla \varphi}_{L^{2}}^{2} \\
& \quad + \mathcal{Z} \left (D \norm{\nabla \sigma}_{L^{2}}^{2} + K \norm{\sigma}_{L^{2}(\Gamma)}^{2} + \lambda_{c} \int_{\Omega} h(\varphi_{2}) \abs{\sigma}^{2} \dx \right )  \\
& \quad =  \int_{\Omega} B \lambda_{p} (\sigma_{1} h(\varphi_{1}) - \sigma_{2} h(\varphi_{2})) \varphi - B \lambda_{a} (h(\varphi_{1}) - h(\varphi_{2})) \varphi \dx \\
& \quad +  \int_{\Omega} D \eta \mathcal{Z} \nabla \varphi \cdot \nabla \sigma - \lambda_{c} \mathcal{Z} \sigma_{1} (h(\varphi_{1}) - h(\varphi_{2})) \sigma \dx + \mathcal{Z}K \int_{\Gamma} \Sigma_{\infty} \sigma \dHaus \\
& \quad + \int_{\Omega} \mathcal{Y} \mu \varphi - A (\Psi'(\varphi_{1}) - \Psi'(\varphi_{2})) (\varphi \mathcal{Y} - \mu) + \chi_{\varphi} \sigma (\mathcal{Y} \varphi - \mu ) \dx,
\end{aligned}
\end{equation}
where we have used the splitting
\begin{align*}
(\sigma_{1} h(\varphi_{1}) - \sigma_{2} h(\varphi_{2})) \sigma = \abs{\sigma}^{2} h(\varphi_{2}) + \sigma_{1} (h(\varphi_{1}) - h(\varphi_{2})) \sigma.
\end{align*}
By H\"{o}lder's inequality, Young's inequality and the Sobolev embedding \eqref{Sobo:H1L4}, we find that the first line on the right-hand side of \eqref{uniq:quasi:energyidentity} can be estimated as
\begin{equation}\label{uniq:quasi:RHS:1}
\begin{aligned}
& \int_{\Omega} B \lambda_{p} (\sigma_{1} h(\varphi_{1}) - \sigma_{2} h(\varphi_{2})) \varphi - B \lambda_{a} (h(\varphi_{1}) - h(\varphi_{2})) \varphi \dx \\
& \quad \leq B \lambda_{p} \norm{\sigma_{1}}_{L^{2}} \mathrm{L}_{h} \norm{\varphi}_{L^{4}}^{2} + B \lambda_{p} h_{\infty} \norm{\sigma}_{L^{2}} \norm{\varphi}_{L^{2}} + B \lambda_{a} \mathrm{L}_{h} \norm{\varphi}_{L^{2}}^{2} \\
& \quad \leq \left ( B \lambda_{p} \mathrm{L}_{h} \norm{\sigma_{1}}_{L^{\infty}(0,T;L^{2})} C_{\mathrm{S}}^{2} + B \lambda_{a} \mathrm{L}_{h} + B^{2} \lambda_{p}^{2} h_{\infty}^{2} \right ) \norm{\varphi}_{L^{2}}^{2}  \\
& \quad + B \lambda_{p} \mathrm{L}_{h} \norm{\sigma_{1}}_{L^{\infty}(0,T;L^{2})} C_{\mathrm{S}}^{2} \norm{\nabla \varphi}_{L^{2}}^{2} + \frac{1}{4} \norm{\sigma}_{L^{2}}^{2}.
\end{aligned}
\end{equation}
Meanwhile, the second line on the right-hand side of \eqref{uniq:quasi:energyidentity} can be estimated as
\begin{equation}\label{uniq:quasi:RHS:2}
\begin{aligned}
& \int_{\Omega} D \eta \mathcal{Z} \nabla \varphi \cdot \nabla \sigma - \lambda_{c} \mathcal{Z} \sigma_{1} (h(\varphi_{1}) -  h(\varphi_{2})) \sigma \dx + \mathcal{Z} K \int_{\Gamma} \Sigma_{\infty} \sigma \dHaus \\
& \quad \leq \frac{\mathcal{Z} D}{2} \norm{\nabla \sigma}_{L^{2}}^{2} + \frac{\mathcal{Z} D \eta^{2}}{2} \norm{\nabla \varphi}_{L^{2}}^{2} + \mathcal{Z} \lambda_{c} \left (\norm{\sigma_{1}}_{L^{4}} \mathrm{L}_{h} \norm{\varphi}_{L^{4}} \norm{\sigma}_{L^{2}} \right ) \\
& \quad + \frac{\mathcal{Z} K}{2} \norm{\Sigma_{\infty}}_{L^{2}(\Gamma)}^{2} + \frac{\mathcal{Z} K}{2} \norm{\sigma}_{L^{2}(\Gamma)}^{2} \\
& \quad \leq \frac{\mathcal{Z} D}{2} \norm{\nabla \sigma}_{L^{2}}^{2} + \left ( \frac{\mathcal{Z} D \eta^{2}}{2} + \mathcal{Z}^{2} \lambda_{c}^{2} \norm{\sigma_{1}}_{L^{\infty}(0,T;H^{1})}^{2} C_{\mathrm{S}}^{4} \mathrm{L}_{h}^{2} \right ) \norm{\nabla \varphi}_{L^{2}}^{2} \\
& \quad + \mathcal{Z}^{2} \mathrm{L}_{h}^{2} \lambda_{c}^{2} \norm{\sigma_{1}}_{L^{\infty}(0,T;H^{1})}^{2} C_{\mathrm{S}}^{4}  \norm{\varphi}_{L^{2}}^{2} +  \frac{1}{4}  \norm{\sigma}_{L^{2}}^{2} \\
& \quad + \frac{\mathcal{Z} K}{2} \norm{\Sigma_{\infty}}_{L^{2}(\Gamma)}^{2} + \frac{\mathcal{Z} K}{2} \norm{\sigma}_{L^{2}(\Gamma)}^{2}.
\end{aligned} 
\end{equation}
Here we point out that we use the assumption $\sigma_{1} \in L^{\infty}(0,T;H^{1})$.  Similarly, the last term on the right-hand side of \eqref{uniq:quasi:energyidentity} can be estimated as
\begin{equation}\label{uniq:quasi:RHS:3}
\begin{aligned}
& \int_{\Omega} \mathcal{Y} \mu \varphi - A (\Psi'(\varphi_{1}) - \Psi'(\varphi_{2})) (\varphi \mathcal{Y} - \mu) + \chi_{\varphi} \sigma (\mathcal{Y} \varphi - \mu ) \dx \\
& \quad \leq \mathcal{Y}\norm{\mu}_{L^{2}} \norm{\varphi}_{L^{2}} + A \mathrm{L}_{\Psi'} \left (\mathcal{Y} \norm{\varphi}_{L^{2}}^{2} + \norm{\varphi}_{L^{2}} \norm{\mu}_{L^{2}} \right ) \\
& \quad + \chi_{\varphi} \mathcal{Y} \norm{\sigma}_{L^{2}} \norm{\varphi}_{L^{2}} + \chi_{\varphi} \norm{\sigma}_{L^{2}} \norm{\mu}_{L^{2}} \\
& \quad \leq \frac{3}{4} \norm{\mu}_{L^{2}}^{2} + C(A, \mathcal{Y}, \mathrm{L}_{\Psi'}, \chi_{\varphi}) \norm{\varphi}_{L^{2}}^{2} + 2\chi_{\varphi}^{2} \norm{\sigma}_{L^{2}}^{2}.
\end{aligned}
\end{equation}
Substituting \eqref{uniq:quasi:RHS:1}, \eqref{uniq:quasi:RHS:2} and \eqref{uniq:quasi:RHS:3} into \eqref{uniq:quasi:energyidentity} leads to
\begin{align*}
& \frac{\dd}{\dt} \frac{B}{2} \norm{\varphi}_{L^{2}}^{2} + \frac{1}{4} \norm{\mu}_{L^{2}}^{2} + \frac{D \mathcal{Z}}{2} \norm{\nabla \sigma}_{L^{2}}^{2} + \frac{\mathcal{Z} K}{2} \norm{\sigma}_{L^{2}(\Gamma)}^{2} - \mathcal{C}\norm{\varphi}_{L^{2}}^{2}  \\
& \quad +  \norm{\nabla \varphi}_{L^{2}}^{2} \left ( B \mathcal{Y} - \frac{\mathcal{Z} D \eta^{2}}{2} -  B \lambda_{p} \mathrm{L}_{h} \norm{\sigma_{1}}_{L^{\infty}(0,T;L^{2})} C_{\mathrm{S}}^{2} - \mathcal{Z}^{2} \lambda_{c}^{2} \norm{\sigma_{1}}_{L^{\infty}(0,T;H^{1})}^{2} \mathrm{L}_{h}^{2} C_{\mathrm{S}}^{4}  \right ) \\
& \quad -  \norm{\sigma}_{L^{2}}^{2} \left ( 2 \chi_{\varphi}^{2} + \frac{1}{2} \right ) \leq \frac{\mathcal{Z} K}{2} \norm{\Sigma_{\infty}}_{L^{2}(\Gamma)}^{2},
\end{align*}
where we have used the non-negativity of $h(\cdot)$ and $\lambda_{c}$ to neglect the term $\lambda_{c} h(\varphi_{2}) \abs{\sigma}^{2}$, and $\mathcal{C}$ is a positive constant depending on $A$, $B$, $\mathcal{Y}$, $\mathcal{Z}$, $\mathrm{L}_{h}$, $\mathrm{L}_{\Psi'}$, $\chi_{\varphi}$, $\lambda_{p}$, $\lambda_{a}$, $\lambda_{c}$, $C_{\mathrm{S}}$, $\norm{\sigma_{1}}_{L^{\infty}(0,T;H^{1})}$, and $h_{\infty}$.  By \eqref{boundary:Poincare:sigma}, we see that
\begin{align*}
& \frac{1}{2}\mathcal{Z} \left ( D \norm{\nabla \sigma}_{L^{2}}^{2} + K \norm{\sigma}_{L^{2}(\Gamma)}^{2} \right ) -  \left ( 2 \chi_{\varphi}^{2} + \frac{1}{2} \right )\norm{\sigma}_{L^{2}}^{2} \\
& \quad \geq \left (  \frac{1}{2}\mathcal{Z} \min (D, K)  - 2 C_{\mathrm{P}}^{2} \left ( 2 \chi_{\varphi}^{2} + \frac{1}{2} \right ) \right ) \left (\norm{\nabla \sigma}_{L^{2}}^{2} + \norm{\sigma}_{L^{2}(\Gamma)}^{2} \right ),
\end{align*}
and so in choosing 
\begin{align*}
\mathcal{Z} & > \frac{4 C_{\mathrm{P}}^{2}}{\min(D,K) } \left ( 2 \chi_{\varphi}^{2} + \frac{1}{2} \right ), \\
\mathcal{Y} & > \frac{1}{B} \left ( \frac{\mathcal{Z} D \eta^{2}}{2} +  B \lambda_{p} \mathrm{L}_{h} \norm{\sigma_{1}}_{L^{\infty}(0,T;L^{2})} C_{\mathrm{S}}^{2} + \mathcal{Z}^{2} \lambda_{c}^{2} \norm{\sigma_{1}}_{L^{\infty}(0,T;H^{1})}^{2} \mathrm{L}_{h}^{2} C_{\mathrm{S}}^{4}  \right ),
\end{align*}
we find that there exist constants $\mathcal{C}, \overline{c} > 0$ such that
\begin{align*}
\frac{\dd}{\dt} \norm{\varphi}_{L^{2}}^{2} - \mathcal{C} \norm{\varphi}_{L^{2}}^{2} + \norm{\mu}_{L^{2}}^{2} + \norm{\nabla \sigma}_{L^{2}}^{2} + \norm{\sigma}_{L^{2}(\Gamma)}^{2} + \norm{\nabla \varphi}_{L^{2}}^{2} \leq \overline{c} \norm{\Sigma_{\infty}}_{L^{2}(\Gamma)}^{2},
\end{align*}
and a similar argument to Section \ref{sec:Uniqueness} yields the desired result.

\section{Discussion}\label{sec:discussion}
We point out that we are not able to improve our class of admissible potentials to those with polynomial growth of order higher than $2$.  In particular, our well-posedness results do not cover the case of the classical quartic double-well potential.  This is due to the fact that in the derivation of \eqref{apriori:main} (specifically in \eqref{apriori:musourceterms}), we encounter a term of the form
\begin{align}
\norm{\mu}_{L^{2}} \left (1 + \norm{\sigma}_{L^{2}} \right ). 
\end{align}
If we use the equation for the chemical potential, this leads to a term of the form
\begin{align}\label{crosstermPsiprime}
\norm{\Psi'(\varphi)}_{L^{2}} \left (1 + \norm{\sigma}_{L^{2}} \right ) .
\end{align}
If $\Psi'$ has polynomial growth of order $q$, i.e., $\abs{\Psi'(t)} \leq R(1 + \abs{t}^{q})$ for some positive constant $R$ and for all $t \in \R$, then we have to control the product
\begin{align*}
\norm{\varphi}_{L^{2q}}^{q} \left (1 + \norm{\sigma}_{L^{2}} \right )
\end{align*}
with the $H^{1}$-norms of $\varphi$ and $\sigma$.  In the absence of any a priori bounds before \eqref{apriori:main}, we have to consider $q = 1$, that is, $\Psi$ has at most quadratic growth.  

This differs from the analysis of \cite{article:ColliGilardiHilhorst15, article:FrigeriGrasselliRocca15}, where the Lyapunov-type energy identity \eqref{Rocca:energy} automatically gives a first a priori estimate without the need to estimate the square of the mean of $\mu$, or equivalently an estimate on $\norm{\Psi'(\varphi)}_{L^{2}}$, which is present in our setting.  Instead of \eqref{assump:Psi}, we may also consider potentials that satisfy
\begin{align}\label{Psiprimesqrt}
\abs{\Psi'(s)} \leq k_{1} \sqrt{\Psi(s)} + k_{2},
\end{align}
for positive constants $k_{1}$ and $k_{2}$.  This yields
\begin{align*}
\norm{\Psi'(\varphi)}_{L^{2}} \leq k_{1} \norm{\sqrt{\Psi(\varphi)}}_{L^{2}} + k_{2} \abs{\Omega} \leq k_{1} \abs{\Omega}^{\frac{1}{2}} \norm{\Psi(\varphi)}_{L^{1}}^{\frac{1}{2}} + k_{2} \abs{\Omega}.
\end{align*}
This allows us to estimate \eqref{crosstermPsiprime} using $\norm{\Psi(\varphi)}_{L^{1}}$ instead of relying on any growth assumptions on $\Psi'$.  However, a scaling argument with $\Psi(s) \sim \abs{s}^{r}$ shows that \eqref{Psiprimesqrt} is satisfied only if $r \leq 2$.  Thus, we do not gain much if we replace $\eqref{assump:Psi}_{2}$ with \eqref{Psiprimesqrt}.  Moreover, \eqref{Psiprimesqrt} seems to be a more restrictive assumption than $\eqref{assump:Psi}_{2}$.

Lastly, we note that \cite[Lemma 2]{article:FrigeriGrasselliRocca15} provides an approximation procedure to potentials with polynomial growth of order $6$ by a sequence of regular potentials with quadratic growth.  This is accomplished by means of a Yosida regularisation of the derivative $\Psi'$.  However, we are not able to apply this idea to our analysis as the key priori estimate \eqref{apriori:main} is not uniform in the constant $R_{3}$, which acts as the regularisation parameter in the corresponding Yosida approximation. 

\section{Conclusion}
In this work, we provide well-posedness results for a system coupling a Cahn--Hilliard equation and a parabolic reaction-diffusion equation to model tumour growth with chemotaxis and active transport.  The existence of weak solutions is shown using a Galerkin procedure.  In contrast to some diffuse interface models for tumour growth studied in the literature, the model presented here admits an energy equality with non-dissipative right-hand sides and allows for some realistic source terms.  The presence of the source terms places some restrictions on the class of admissible potentials, namely potentials with quadratic growth.  In addition, we also study a system coupling a Cahn--Hilliard equation and an elliptic equation, which is realistic when bulk diffusion of the nutrient is fast and is often the case in applications.  We are also able to prove the continuous dependence on initial and boundary data for the chemical potential $\mu$ in $L^{2}(\Omega \times (0,T))$ and for the order parameter $\varphi$ in $L^{\infty}(0,T;L^{2}(\Omega))$.

\bibliographystyle{plain}
\bibliography{GLCHNNeumann}

\begin{thebibliography}{10}

\bibitem{book:Alt}
H.W. Alt.
\newblock {\em {Lineare Funktionalanalysis: Eine anwendungsorientierte
  Einf\"{u}hrung}}.
\newblock Springer-Verlag Berlin Heidelberg, 2006.

\bibitem{preprint:BosiaContiGrasselli14}
S.~Bosia, M.~Conti, and M.~Grasselli.
\newblock {On the Cahn--Hilliard--Brinkman system}.
\newblock {\em Commun. Math. Sci.}, 13(6):1541--1567, 2015.

\bibitem{incoll:ByrneCancerbook}
H.M. Byrne.
\newblock {Modelling Avascular Tumour Growth}.
\newblock In L.~Preziosi, editor, {\em Cancer Modelling and Simulation, Chapman
  \& Hall/CRC Mathematical and Computational Biology}. CRC Press, 2003.

\bibitem{article:ColliGilardiHilhorst15}
P.~Colli, G.~Gilardi, and D.~Hilhorst.
\newblock {On a Cahn--Hilliard type phase field model related to tumor growth}.
\newblock {\em Discrete Contin. Dyn. Syst.}, 35(6):2423--2442, 2015.

\bibitem{article:ColliGilardiRoccaSprekelsVV}
P.~Colli, G.~Gilardi, E.~Rocca, and J.~Sprekels.
\newblock {Vanishing viscosities and error estimate for a Cahn--Hilliard type
  phase field system related to tumor growth}.
\newblock {\em Nonlinear Anal. Real World Appl.}, 26:93--108, 2015.

\bibitem{article:ColliGilardiRoccaSprekelsAA}
P.~Colli, G.~Gilardi, E.~Rocca, and J.~Sprekels.
\newblock {Asymptotic analyses and error estimates for a Cahn--Hilliard type
  phase field system modelling tumor growth}.
\newblock To appear in Discrete Contin. Dyn. Syst. Ser. S, 2016.

\bibitem{article:CristiniLiLowengrubWise09}
V.~Cristini, X.~Li, J.S. Lowengrub, and S.M. Wise.
\newblock {Nonlinear simulations of solid tumor growth using a mixture model:
  invasion and branching}.
\newblock {\em J. Math. Biol.}, 58:723--763, 2009.

\bibitem{article:FengWise12}
X.~Feng and S.~Wise.
\newblock {Analysis of a Darcy--Cahn--Hilliard diffuse interface model for the
  Hele--Shaw flow and its fully discrete finite element approximation}.
\newblock {\em SIAM J. Numer. Anal.}, 50(3):1320--1343, 2012.

\bibitem{article:FrigeriGrasselliRocca15}
S.~Frigeri, M.~Grasselli, and E.~Rocca.
\newblock {On a diffuse interface model of tumor growth}.
\newblock {\em European J. Appl. Math.}, 26:215--243, 2015.

\bibitem{article:GarckeLamDirichlet}
H.~Garcke and K.F. Lam.
\newblock {Analysis of a Cahn--Hilliard system with non zero Dirichlet
  conditions modelling tumour growth with chemotaxis}.
\newblock arXiv:1604.00287, 2016.

\bibitem{article:GarckeLamSitkaStyles}
H.~Garcke, K.F. Lam, E.~Sitka, and V.~Styles.
\newblock {A Cahn--Hilliard--Darcy model for tumour growth with chemotaxis and
  active transport}.
\newblock {\em Math. Models Methods Appl. Sci.}, 26(6):1095--1148, 2016.

\bibitem{article:HawkinsZeeOden12}
A.~Hawkins-Daarud, K.G. van~der Zee, and J.T. Oden.
\newblock {Numerical simulation of a thermodynamically consistent four-species
  tumor growth model}.
\newblock {\em Int. J. Numer. Methods Biomed. Eng.}, 28:3--24, 2012.

\bibitem{article:Kampmann}
D.~Hilhorst, J.~Kampmann, T.N. Nguyen, and K.G. van~der Zee.
\newblock {Formal asymptotic limit of a diffuse-interface tumor-growth model}.
\newblock {\em Math. Models Methods Appl. Sci.}, 25(6):1011--1043, 2015.

\bibitem{article:JiangWuZheng14}
J.~Jiang, H.~Wu, and S.~Zheng.
\newblock {Well-posedness and long-time behavior of a non-autonomous
  Cahn--Hilliard--Darcy system with mass source modeling tumor growth}.
\newblock {\em J. Differential Equ.}, 259(7):3032--3077, 2015.

\bibitem{article:LeeLowengrubGoodman01}
H.~Lee, J.~Lowengrub, and J.~Goodman.
\newblock {Modeling pinchoff and reconnection in a Hele--Shaw cell. I. The
  models and their calibration}.
\newblock {\em Phys. Fluids}, 14(2):492--513, 2002.

\bibitem{article:LeeLowengrubGoodman01:Part2}
H.~Lee, J.~Lowengrub, and J.~Goodman.
\newblock {Modeling pinchoff and reconnection in a Hele--Shaw cell. II.
  Analysis and simulation in the nonlinear regime}.
\newblock {\em Phys. Fluids}, 14(2):514--545, 2002.

\bibitem{article:LowengrubTitiZhao13}
J.S. Lowengrub, E.~Titi, and K.~Zhao.
\newblock {Analysis of a mixture model of tumor growth}.
\newblock {\em European J. Appl. Math.}, 24:691--734, 2013.

\bibitem{book:Royden}
H.L. Royden and P.~Fitzpatrick.
\newblock {\em Real Analysis}.
\newblock Featured Titles for Real Analysis Series. Prentice Hall, 4th edition,
  2010.

\bibitem{book:Salsa}
S.~Salsa, F.M.G. Vegnu, A.~Zaretti, and P.~Zunino.
\newblock {\em {A Primer on PDEs. Models, Methods, Simulations}}.
\newblock Springer-Verlag Mailand, 2013.

\bibitem{article:Simon86}
J.~Simon.
\newblock {Compact sets in space $L^{p}(0,T;B)$}.
\newblock {\em Ann. Mat. Pura Appl.}, 146(1):65--96, 1986.

\bibitem{book:Temam}
R.~Temam.
\newblock {\em {Infinite-Dimensional Dynamical Systems in Mechanics and
  Physics}}, volume~68.
\newblock Springer, New York, 1988.

\bibitem{article:WangWu}
X.~Wang and H.~Wu.
\newblock {Long-time behavior for the Hele--Shaw--Cahn--Hilliard system}.
\newblock {\em Asymptot. Anal.}, 78(4):217--245, 2012.

\bibitem{article:WangZhang}
X.~Wang and Z.~Zhang.
\newblock {Well-posedness of the Hele--Shaw--Cahn--Hilliard system}.
\newblock {\em Ann. Inst. H. Poincar\'{e} Anal. Non Lin\'{e}aire},
  30(3):367--–384, 2013.

\end{thebibliography}
\end{document}